\documentclass[11 pt, twoside]{amsart}

\usepackage{amsthm,amssymb,amsmath,amsfonts,mathrsfs,amscd}
\usepackage[latin1]{inputenc}
\usepackage{enumitem}
\usepackage{verbatim}
\usepackage{leftidx}
\usepackage[all,2cell]{xy}
\usepackage[margin=1in, a4paper]{geometry}
\usepackage{calligra} 
\DeclareMathAlphabet{\mathcalligra}{T1}{calligra}{m}{n} \DeclareFontShape{T1}{calligra}{m}{n}{<->s*[2.2]callig15}{} 

\newcommand{\sr}{\ensuremath{\mathcalligra{r}}}

\newtheorem{thm}{Theorem}

\numberwithin{equation}{subsection}
\newtheorem{theorem}[equation]{Theorem}

\newtheorem{lemma}[equation]{Lemma}
\newtheorem{proposition}[equation]{Proposition}
\newtheorem{cor}[equation]{Corollary}

\theoremstyle{definition}

\theoremstyle{definition}
\newtheorem{definition}[equation]{Definition}
\newtheorem{remark}[equation]{Remark}

\newcommand{\fr}{\mathfrak}

\newcommand{\sbt}{\,\begin{picture}(-1,1)(-1,-3)\circle*{2}\end{picture}\ }

\def \A {\mathbb{A}}
\def \Q{\mathbb{Q}}

\def \G {\mathbb{G}}

\def \C {\mathbb{C}}

\def \R {\mathbb{R}}

\def\Z{\mathbb Z}

\def \calf {\mathcal{F}}

\def \calh {\mathcal{H}}

\def \cala {\mathcal{A}}

\def \calc {\mathcal{C}}
\def \calo {\mathcal{O}}
\def \cals {\mathcal{S}}
\def \calw {\mathcal{W}}
\def \calm {\mathcal{M}}
\def \tcalm {\widetilde{\mathcal{M}}}

\def \HGK {\mathcal{H}^G_{K_f}}

\def\id{1\!\!1}

\def\Hom{\mathrm{Hom}}

\def\Sym{\mathrm{Sym}}

\def\Ad{\mathrm{Ad}}

\def\GL{\mathrm{GL}}
\def\GO{\mathrm{GO}}
\def\Orth{\mathrm{O}}
\def\SO{\mathrm{SO}}

\def\g{\mathfrak{g}}
\def\k{\mathfrak{k}}

\def\Aut{\mathrm{Aut}}

\def\Res{\mathrm{Res}}
\def\det{\mathrm{det}}
\def\dim{\mathrm{dim}}
\def\End{\mathrm{End}}

\def\ui{{\underline i}}
\def\uj{{\underline j}}

\def\tcstar{{\fr t^c}^*}

\def\SGK{S^G_{K_f}}
\def\Hbar{\bar H}
\def\sv{{\sf v}}

\def\ram{\mathrm{ram}}
\def\Lalg{L^{\mathrm{alg}}}
\def\cart{\underset{\text{c}}{\times}}
\def\prv{\underset{\text{prv}}{\times}}

\title[Special values of adjoint $L$-functions and congruences]{Special values of adjoint $L$-functions and congruences for automorphic forms on $\GL(n)$ over a number field}
\author{\bf Baskar Balasubramanyam \ \ \and \ \ A. Raghuram}

\address{Department of Mathematics, Indian Institute of Science Education and Research, 
Dr. Homi Bhabha Road, Pashan, Pune, Maharashtra 411008, INDIA.} 
\email{baskar@iiserpune.ac.in}
\email{raghuram@iiserpune.ac.in}

\date{\today}

\begin{document}

\maketitle

\tableofcontents

\section{\bf Introduction}
In the 80's Hida \cite{Hi1} proved that if a prime $p$, taken to be outside an explicit finite set of exceptional primes, divides the algebraic part of the value at $s=1$ of the adjoint $L$-function attached to a holomorphic primitive cusp form $f = \sum a_n(f)q^n$, then $p$ is a congruence prime, i.e., there is another primitive cusp form
$g = \sum a_n(g) q^n$ of the same weight and level as $f$ such that $a_n(f) \equiv a_n(g) \pmod{p}.$ A converse to 
such a result, that congruence primes are the primes which appear in adjoint $L$-values, 
was proved by Hida \cite{Hi2} in the ordinary case, and by Ribet~\cite{ribet} in the non-ordinary case.  
Later, the results of 
\cite{Hi1} were generalized to various other $\GL(2)$-contexts: Ghate~\cite{Gh} and Dimitrov~\cite{D} considered the Hilbert modular situation; Urban~\cite{U} considered $\GL(2)$ over an imaginary quadratic extension; 
Hida~\cite{hida-shimura-vol}, and very recently Namikawa \cite{N}, dealt with the case of $\GL(2)$ over any number field
$F$ using a mix of classical and adelic language. 
However, the above mentioned converse due to Hida and Ribet is not known for $\GL(2)/F$ if $F \neq \Q$, except for some work by Ghate~\cite{Gh1}.
See also the article by Doi, Hida and Ishii~\cite{doi-hida-ishii} for a discussion of the history of this problem. 
In this article we generalize the results of \cite{D}, \cite{Gh}, \cite{Hi1}, \cite{hida-shimura-vol}, \cite{N}, and \cite{U}  to the case of a cohomological cuspidal automorphic representation $\pi$ of $\GL_n$ over any number field. We prove an algebraicity, and in fact, an integrality result for $L(1, {\rm Ad}^\circ, \pi, \varepsilon).$ See Thm.\,\ref{thm:l-value}. 
Often times, a cohomological interpretation of an analytic theory of $L$-functions, depends on an assumption that a 
quantity coming from archimedean considerations in nonzero. In our situation, we adapt the methods of recent work of 
Sun \cite{Sun} to prove an appropriate nonvanishing result; see Prop.\,\ref{prop:nv-hyp}. 
With the integrality result on the adjoint $L$-value in hand,  we then prove that a prime, outside a finite set of exceptional primes, that divides this $L$-value is a congruence prime. See Thm.\,\ref{thm:congruence}. 

\medskip

We will now describe our results in some more detail. Let $F$ be a number field and let $G$ denote the restriction of scalars from $F$ to $\Q$ of the algebraic group $\GL_n/F$. Let $\A$ denote the ad\`eles over $\Q$, let $\A_f$ and $\A_\infty$ respectively denote the finite and infinite part of $\A$. Let $K_f$ be an open compact subgroup of $G(\A_f)$. Denote by $S^G_{K_f}$ the locally symmetric space $G(\Q)\backslash G(\A)/K_f K_\infty^\circ$ for a suitable choice of $K_\infty \subset G_\infty = G(\R)$ as in \S\ref{sec:basic-setup}, and $K_\infty^\circ$ is the connected component containing the identity in $K_\infty$. Let $\varepsilon$ be a character on $K_\infty/K_\infty^\circ$ and we will let $\tilde \varepsilon = (-1)^{n-1} \varepsilon$.
Let $\pi$ be a cuspidal automorphic representation of $G(\A)$ with central character $\omega$. Let $\tilde \pi$ be the contragradient of $\pi.$ The associated Rankin--Selberg $L$-function $L(s, \pi \times \tilde \pi)$ has a meromorphic continuation to all of $\C$ with a simple pole at $s=1$. We define the adjoint $L$-function of $\pi$ by the formula
$$
L(s, \pi \times \tilde \pi) \ = \  \tilde\zeta_F (s) L(s, \Ad^0, \pi),
$$
where $\tilde\zeta_F(s)$ is the completed Dedekind zeta function of $F.$ 
Taking residues at $s=1$ gives us an expression for the special value $L(1, \Ad^0, \pi)$ in terms of the residue Res$_{s=1} L(s, \pi \times \tilde \pi)$. This residue can be written in terms of a Petersson inner product of automorphic forms in $\pi$ and $\tilde \pi$. Now we suppose that $\pi$ is of cohomological type, i.e., it contributes to the cuspidal cohomology of $S^G_{K_f}$ with coefficients in a sheaf attached to an algebraic representation of $G$ of highest weight 
$\lambda$; we will write this as $\pi \in {\rm Coh}(G, K_f, \lambda).$ In such a situation, the finite part of $\pi$ is defined over a number field denoted $\Q(\pi)$ called the rationality field of $\pi.$ 
Then this inner product has an algebraic description in terms of Poincar\'e duality for the cohomology of the locally symmetric space $S^G_{K_f}$. This gives us our first theorem, proved in \S \ref{proof-thma}:

\begin{thm}
Let $\pi \in {\rm Coh}(G, K_f, \lambda)$ and let $\varepsilon$ be a character of $K_\infty/K_\infty^\circ$ that is 
`permissible' for $\pi$ as in \S\ref{sec:whit-coh}. Then there exist nonzero complex numbers: 
\begin{itemize}
\item $\omega_F$ depending only on the base field $F$, 
\item $\fr p_{\rm ram} (\pi)$ depending only on the ramified local components of $\pi$, 
\item $\fr p_\infty (\pi)$ depending only on the archimedean components of $\pi$, 
\item $\fr p^\varepsilon (\pi)$ (resp., $\fr q^{\tilde{\varepsilon}}(\tilde \pi)$) coming from a comparison of $\Q(\pi)$-structure on a Whittaker model of $\pi_f$ and a $\Q(\pi)$-structure on a realization of $\pi_f$ in bottom (resp., top) degree cuspidal cohomology, 
\end{itemize}
such that the quantity 
$$
L^{\mathrm{alg}} (1, \Ad^0, \pi , \varepsilon) := \frac{L(1, \Ad^0, \pi)}{\omega_F \cdot \fr p_{\rm ram} (\pi) \cdot \fr p_\infty (\pi) \cdot \fr p^\varepsilon (\pi) \cdot \fr q^{\tilde \varepsilon} (\tilde \pi)}
$$
is algebraic. Moreover, for all $\sigma \in \mathrm{Aut}(\C)$ we have
$$
\sigma( \Lalg (1, \Ad^0, \pi, \varepsilon)) \ = \  \Lalg (1, \Ad^0, {^\sigma} \pi, \varepsilon).
$$
In particular, $\Lalg (1, \Ad^0, \pi, \varepsilon)  \in \Q(\pi).$ 

\medskip

For a prime $p$, take an extension $E$ of $\Q_p$ that contains $F$, the rationality field 
$\Q(\pi)$, and all their conjugates. Fix an isomorphism $\iota : \overline{\Q}_p \to \C$ of an algebraic closure of 
$\Q_p$ with $\C.$ Let $\calo$ be the ring of integers of $E.$ We can canonically refine the definitions of the periods 
$\fr p^\varepsilon (\pi)$ and $\fr q^{\tilde{\varepsilon}}(\tilde \pi)$ as in (\ref{eqn:refined-periods}) so that $\Lalg (1, \Ad^0, \pi, \varepsilon) \in \iota(\calo).$ 
\end{thm}

This theorem is proved by giving a cohomological interpretation to Rankin--Selberg integrals for $\GL_n \times \GL_n$ in the special situation when we have a pair $\pi \times \tilde\pi$ of a representation and its contragredient; 
see the diagram in \S\,\ref{sec:main-idea}. 
The periods $\fr p^\varepsilon(\pi)$ and 
$\fr q^{\tilde \varepsilon}(\tilde \pi)$ have been studied in \cite{RS}.  The refinement to get integrality results is also similar to the definition of canonical periods for modular forms by Vatsal~\cite{Vat}. The nonvanishing result Prop.\,\ref{prop:nv-hyp}
for the archimedean quantity mentioned earlier is proved in \S\,\ref{nv-sec}.

We should mention that a form of the above theorem, in the special case when the base field 
$F$ is a CM-field, and when $\pi$ is conjugate-self-dual, has been announced by Grobner, Harris and 
Lapid \cite{GHL}; in their situation, 
the Rankin--Selberg $L$-function splits into a product of two Asai $L$-functions and their work studies the arithmetic properties of such Asai $L$-values.


Let's highlight a phenomenon which is already seen in certain $\GL(2)$ contexts as in Hida~\cite{hida-shimura-vol} and Urban~\cite{U}. The value $L(1, \Ad^0, \pi)$ is a critical value (in the sense of Deligne) if and only if $n=2$ and $F$ is totally real; see Prop.\,\ref{prop:critical}. By the theorems in \cite{raghuram-imrn} and \cite{raghuram-2013} it is clear that 
the periods $\fr p^\varepsilon(\pi)$ arising from Whittaker models and bottom-degree cohomology 
appear in critical values, albeit for $L$-functions for $\GL(n) \times \GL(n-1).$ 
One may expect therefore that $\fr p^\varepsilon(\pi)$ is somehow related to Deligne's periods attached to the (conjectural) motive $M(\pi)$ corresponding to $\pi$; see Grobner--Harris~\cite{GH} for some related results.  
From the above theorem, one may expect that the other period 
$\fr q^{\tilde \varepsilon}(\tilde \pi)$ arising from a comparison of Whittaker model and 
top-degree cuspidal cohomology is a Beilinson type regulator attached to $M(\pi),$ since $L(1, \Ad^0, \pi)$ is not critical in general. However, note that if $\pi$ has a Shalika model, then the period obtained by comparing rational structures on Shalika models and top-degree cohomology turns out to be related to critical values via the results of 
Grobner--Raghuram~\cite{grobner-raghuram-ajm}. It is an interesting problem then to understand the precise motivic interpretation of the various periods arising from top-degree cuspidal cohomology.

\medskip

We will now discuss congruence primes for automorphic forms on $\GL_n$ and primes appearing in adjoint $L$-values. 
Recall that we have a prime $p$, an extension $E$ of $\Q_p$ that contains $F$, $\Q(\pi)$, and all their conjugates, and we have fixed an isomorphism $\iota : \overline{\Q}_p \to \C.$  Much of what follows on congruences depends on the choice of this isomorphism $\iota$, 
however, for brevity, we will suppress it from our notation. 
Let $\wp$ be the maximal ideal of  the ring of integers $\calo$ of $E.$
Let $\pi$ and $\pi^\prime$ be two cuspidal automorphic representations for $G.$ Suppose that $\fr l$ is a prime ideal of $F$ away from the primes above $p$ and the ramified primes of $\pi$ and $\pi^\prime$. We denote the Satake parameters of $\pi$ and $\pi^\prime$ at $\fr l$ by $\alpha_{\fr l, 1}, \dots, \alpha_{\fr l, n}$ and $\alpha^\prime_{\fr l, 1}, \dots, \alpha^\prime_{\fr l, n},$ respectively. 
Suppose that $E$ is large enough to contain the fields $\Q(\pi)$ and $\Q(\pi^\prime)$. We say that $\pi$ is congruent to $\pi^\prime$ modulo $\wp$ if for every $\fr l$ above and $1 \le j \le n$, we have
\begin{equation*}
\sum_{i_1<i_2 < \cdots < i_j} \alpha_{\fr l, i_1} \cdots \alpha_{\fr l, i_j} 
\ \equiv \ 
\sum_{i_1<i_2 < \cdots < i_j} \alpha^\prime_{\fr l, i_1} \cdots \alpha^\prime_{\fr l, i_j} \quad \pmod{\wp}. 
\end{equation*}

\medskip

Our second main theorem, proved in \S \ref{proof-thmb}, says that 
\begin{thm}
\label{intro-thm:congruence}
Let $\pi$ and $\varepsilon$ be as above. There exist finite sets $S_1, S_2$ and $S_3$ consisting of rational primes such that if
$$ 
v_\wp (\Lalg (1, \Ad^0, \pi, \varepsilon)) > 0,
$$
\begin{enumerate}
\item and if $p \not \in S_1$, then there exists $\pi^\prime$ congruent to $\pi$ mod $\wp$ and 
$\pi^\prime \not \simeq \pi$,
\item and if $ p \not \in S_2$ and if $\pi$ is of parallel weight, then there exists $\pi^\prime$ congruent to $\pi$ mod $\wp$ and $\pi^\prime \not \simeq {^\sigma} \pi$ for any $\sigma \in \mathrm{Aut}(\C)$,
\item and if $ p \not \in S_3$ and if $\pi$ is of parallel weight, then there exists $\sigma \in \mathrm{Aut} (\C)$ with 
$\pi' = {^\sigma} \pi$ congruent to $\pi$ mod $\wp$ and $\pi^\prime \not \simeq \pi$.
\end{enumerate}
A priori, we can only say that $\pi^\prime$ contributes to the inner cohomology. If we further assume that the highest weight $\lambda$ is regular, then $\pi^\prime$ is cuspidal.
\end{thm}

The set $S_1$ consists of primes which support the torsion classes in the cohomology of the boundary in the 
Borel--Serre compactification of $S^G_{K_f}$; see (\ref{eqn:S1}). The reader is referred to Sect.\,\ref{excluded-primes} where we discuss a situation (modulo well-known expectations in the arithmetic theory of automorphic forms) when the set $S_1$ is possibly an empty set. The sets $S_2$ and $S_3$ are described in 
(\ref{eqn:S2}) and (\ref{eqn:S3}), respectively.

\bigskip

{\small 
{\it Acknowledgements:} 
We thank Haruzo Hida for explaining to us the history of the problem concerning congruences and adjoint $L$-values, and we thank Binyong Sun for his comments on Sect.\,\ref{nv-sec}. 
The second author is grateful to G\"unter Harder for innumerable discussions on number-theoretic applications of cohomology of arithmetic groups, and in particular, on explanations about the relation 
between special values of $L$-functions and the study of congruences, much of which finds its way into this article. 
}

\bigskip
\section{\bf Special values of the adjoint $L$-function}

Let $F$ be a number field of degree $d$ over $\Q$ and let $\calo_F$ be the ring of integers in $F$. Let $r_1$ and $r_2$ denote the number of real and complex embeddings of $F$ respectively. For any place $v$ of $F$, let $F_v$ be the completion of $F$ at $v$. For non-archimedean places, let $\calo_{F,v}$ denote the valuation ring of $F_v$. Let $\A$ be the ad\`eles over $\Q$ and let $\A_f\  (\text{resp.}, \A_\infty)$ denote the finite (resp., infinite) part of the ad\`eles. Let $\A_F = \A \otimes F$, the adele ring of $F$.
Let $G_0$ denote the algebraic group $\GL_n/F$. We denote by $B_0$, $N_0$, $T_0$ and $Z_0$ the standard Borel subgroup of all upper-triangular matrices, the maximal unipotent subgroup in $B_0$, the maximal torus in $B_0,$ and the center of $G_0$, respectively. 
Let $G = {\rm Res}_{F/\Q}(G_0)$ be the restriction of scalars from $F$ to $\Q$ of $G_0$. Similarly, define $B$, $N$,  $T$ and $Z$ as restriction of scalars of the corresponding groups over $F$. Let $S \simeq \G_m$ be the maximal $\Q$-split torus in $Z$.

We will denote by $\cala^\infty (\omega)$, the space of smooth automorphic forms with central character $\omega$ and 
$\cala_0^\infty (\omega)$ denotes the subspace of cusp forms. Let $\varphi \in \cala_0^\infty(\omega)$ and $\varphi^\prime \in \cala_0^\infty(\omega^{-1})$, then we define the Petersson inner product of these forms as
$$ \langle \varphi, \varphi^\prime \rangle = \int_{Z(\A) G(\Q) \backslash G(\A)} \varphi(g) \varphi^\prime (g) \ dg. $$
Note that this integral is well-defined since the integrand is invariant under $Z(\A)$. The measure $dg$ is a product of local measures, normalized as in Jacquet--Shalika~\cite{JS} at finite places, and as in 
Jacquet~\cite{jacquet-arch} at infinite places.

\medskip
\subsection{Rankin--Selberg integrals for $\GL_n \times \GL_n$}

Let $(\pi, V_\pi)$ and $(\pi^\prime, V_{\pi^\prime})$ be cuspidal automorphic representations of $\GL_n/F$. 
Say, $V_\pi \subset \cala_0^\infty (\omega)$ and $V_{\pi^\prime} \subset \cala_0^\infty (\omega^\prime)$ for central characters $\omega$ and $\omega^\prime$. Suppose that $\omega \omega^\prime$ is a unitary character.
We briefly review the Rankin--Selberg theory for the $L$-function associated to $\pi \times \pi^\prime$; the reader is referred to Jacquet--Shalika~\cite{JS} or Cogdell~\cite{Cog} for more details. 

Let $\cals(\A_F^n)$ denote the space of Schwartz--Bruhat functions on $\A_F^n$.  There is an action of $\GL_n(\A_F)$ on $\cals(\A_F^n)$ defined by $(g\cdot \Phi)(x) = \Phi(xg)$, for $\Phi \in \cals(\A_F^n)$ and $g\in \GL_n(\A_F)$. 
Furthermore, for $a \in \A_F^\times$, define a theta series: 
\begin{equation*} 
\Theta_\Phi(a,g) \ = \ \sum_{\xi \in F^n} (g\cdot \Phi) (a \xi) \ = \ \sum_{\xi \in F^n} \Phi(a\xi g). 
\end{equation*}
For a unitary character $\eta : F^\times \backslash \A_F^\times \to \C^\times$, define an Eisenstein series: 
$$
E(g,s)  \ = \ E(g,s,\Phi, \eta) \ = \ 
|\det \ g|^s \int_{F^\times \backslash \A_F^\times} \Theta^\prime_\Phi (a,g) \eta (a) |a|^{ns} \ d^\times a, 
$$ 
where $\Theta^\prime_\Phi (a,g) = \Theta_\Phi(a,g) - \Phi(0)$.
For cusp forms $\varphi \in V_\pi$ and $\varphi^\prime \in V_{\pi^\prime}$, define a Rankin--Selberg integral: 
$$
D(s, \varphi, \varphi^\prime, \Phi) = \int_{Z(\A) G(\Q) \backslash G(\A)} 
\varphi (g) \varphi^\prime(g) E(g,s, \Phi, \omega \omega^\prime) \, dg.
$$
This function is meromorphic in $s\in\C$ and has at most  simple poles at $s=i\sigma$ and $s=1+i\sigma$ when 
$\pi^\prime \cong \tilde \pi \otimes |\det|^{-i\sigma}$ for $\sigma \in \R$, where $\tilde \pi$ is the contragredient of $\pi.$ 
We are interested in the special case when $\pi^\prime = \tilde \pi$, and would like to 
compute the residue at $s=1.$ For $\tilde \varphi \in V_{\tilde \pi}$ we have
$$
\Res_{s=1} D(s, \varphi, \tilde \varphi, \Phi) \ = \ 
\int_{Z(\A) G(\Q) \backslash G(\A)} \varphi (g)\tilde \varphi (g)\ \Res_{s=1} E(g,s, \Phi, \id) \, dg, 
$$
where $\id$ denotes the trivial character.  

Fix  a non-trivial character $\psi$ on $\A_F/F$, and let $\hat \Phi$ be the Fourier transform of $\Phi$, 
$$
\hat \Phi (x) = \int_{\A_F^n} \Phi (y) \psi (\langle x,y\rangle) \ dy. 
$$ 
Applying the Poisson summation formula, which says $\sum_{\xi\in F^n} \Phi (\xi) = \sum_{\xi \in F^n} \hat \Phi (\xi)$,  
to the function $x \mapsto \Phi(ax)$ gives: 
$$ 
\sum_{\xi\in F^n} \Phi (a\xi) = |a|^{-n} \sum_{\xi \in F^n} \hat \Phi (a^{-1} \xi).
$$
Hence,
$$ \Theta^\prime_\Phi (a,g) = |a|^{-n} \sum_{\xi \not = 0} \widehat{(g\cdot \Phi)} (a^{-1} \xi) + |a|^{-n} \widehat{(g\cdot \Phi)} (0) - \Phi(0).$$

Let $F_0$ and $F_1$ be two continuous positive functions on $\R_+^\times$ satisfying the following conditions: $F_0 + F_1 =1$, $F_1(t) = F_0(t^{-1})$ and there exists $0<t_0<1<t_1$ such that $F_0(t) = 0$ for $t<t_0$ and $F_0(t)=1$ for $t>t_1$. Let 
$$
\theta^i (s, \eta, \Phi) =  \int_{F^\times \backslash \A_F^\times} \sum_{\xi \not = 0} \Phi (a \xi) |a|^{ns} \eta(a) F_i(|a|)\ d^\times a, \mbox{ and}
$$
$$
\lambda (s, \eta) = \int_{F^\times \backslash \A_F^\times} |a|^{ns} \eta(a) F_1(|a|)\ d^\times a.
$$
Then,
\begin{eqnarray*} 
E(g,s) &=& |\det \ g|^s \int_{F^\times \backslash \A_F^\times} \sum_{\xi \not = 0} (g \cdot \Phi) (a \xi)  \eta (a) |a|^{ns} \ d^\times a, \\
 &=& |\det \ g|^s \ [\theta^0 (s, \eta, g\cdot \Phi) + \theta^1(s, \eta, g\cdot \Phi)].
\end{eqnarray*}
But,
$$\theta^1(s,\eta, g\cdot \Phi) = \theta^0( n-ns, \eta^{-1}, \widehat{g\cdot \Phi}) - \lambda (n-ns, \eta^{-1}) \widehat{(g\cdot \Phi)}(0) - \lambda(ns, \eta) \Phi(0).$$

In order to calculate residues, note that $\theta^0$ is holomorphic and that $\lambda (s,\id)$ is meromorphic with a simple pole at $s=0.$ (See Godement--Jacquet~\cite[\S 11]{GJ}.) We see that $\widehat{(g\cdot \Phi)} (0) = |\det (g)|^{-1} \hat \Phi (0)$, and hence
$$
\Res_{s=1} E(g,s, \Phi, \id) \ = \ \Res_{s=0} \lambda (s, \id) \cdot \hat \Phi (0).
$$
Since this quantity is independent of $g,$ and  $\Res_{s=0} \lambda (s,\id) = \frac{\mathrm{vol}(F^\times \backslash \A^1_F)}{n}$ (see Zhang~\cite[\S 3.1]{zhang}), we get
\begin{equation}\label{D-residue}
\begin{split}
\Res_{s=1} D(s, \varphi, \tilde \varphi, \Phi) & \ = \  \frac{\mathrm{vol}(F^\times \backslash \A^1_F) \cdot \hat \Phi (0)}{n} \cdot \int_{Z(\A) G(\Q) \backslash G(\A)} \varphi(g) \tilde \varphi (g) \ dg \\
& \ = \ \frac{\mathrm{vol}(F^\times \backslash \A^1_F) \cdot \hat \Phi (0)}{n}  \cdot \langle \varphi, \tilde \varphi \rangle.
\end{split}
\end{equation}

\medskip

\subsection{An integral representation of $L(1, {\rm Ad}^0, \pi)$}

Let $(\pi, V_\pi)$ and $(\tilde \pi, V_{\tilde \pi})$ be as above.  The Rankin--Selberg integrals have good analytic properties, however, they are not Eulerian. To see their relation with $L$-functions, we introduce Whittaker models. 
Recall, that we have fixed a non-trivial character 
$\psi : F\backslash \A_F \to \C^\times.$ A cuspidal automorphic representation of $G(\A)$ is globally generic, i.e., admits 
a Whittaker model. The map taking the $\psi$-Whittaker function of a cusp form gives an isomorphism 
$V_\pi \stackrel{\sim}{\longrightarrow} \calw(\pi, \psi).$ Similarly, we have  
$V_{\tilde \pi} \stackrel{\sim}{\longrightarrow} \calw(\tilde \pi, \psi^{-1}). $
The global Whittaker model decomposes into a restricted tensor product of local Whittaker models: 
 $ \calw (\pi, \psi) \cong \otimes' \calw (\pi_v, \psi_v)$ and 
 $ \calw (\tilde \pi, \psi^{-1}) \cong \otimes' \calw (\tilde \pi_v, \psi^{-1}_v).$ 
 
 At every finite place $v$, let $W_v$ and $\tilde W_v$ be the {\it essential} vectors of Jacquet, Piatetski-Shapiro and Shalika \cite{j-ps-s}. For $v \in S_\infty$, let $W_v$ and $\tilde W_v$ be arbitrary nonzero vectors for now; later these will be taken to be ``cohomological vectors." Choose cusp forms $\varphi \in V_\pi$ and $\tilde \varphi \in V_{\tilde \pi}$ which correspond to a tensor product of these local Whittaker vectors: 
 \begin{equation}
 \label{eqn:special-choice-vectors}
 \varphi \mapsto W = \otimes W_v \quad \text{ and } \quad \tilde \varphi \mapsto \tilde W = \otimes \tilde W_v. 
 \end{equation}
 We take the Schwartz--Bruhat function $\Phi$ to be a tensor product, $\otimes \Phi_v$, of local Schwartz--Bruhat functions. We further assume that for all finite places $v$, the function $\Phi_v$ is the characteristic function of $\calo_v^n$. For the infinite places, we take any $\Phi_v$ such that $\hat \Phi_v (0) \not = 0$.

A standard unfolding argument transforms the Rankin--Selberg integral considered above into a global zeta integral which is Eulerian, i.e.,  
\begin{equation}
\label{eqn:eulerian}
D(s, \varphi, \tilde \varphi, \Phi) \ = \ \prod_v \Psi_v ( s, W_v, \tilde W_v, \Phi_v), \quad \Re(s) \gg 0, 
\end{equation}
where the local zeta integrals are given by:
$$
\Psi_v ( s, W_v, \tilde W_v, \Phi_v) \ = \ \int_{N_0(F_v) \backslash G_0(F_v)} 
W_v(g_v) \tilde W_v(g_v) \Phi_v (e_n g_v) |\det \phantom{!} g_v|^s\ dg_v, 
$$
with $e_n = (0, \dots, 0, 1) \in F^n$. The local zeta integrals converge on a sufficiently large right half-plane and can be meromorphically continued to all of $\C$. When the representations $\pi$ and $\tilde \pi$ are unitary, it is known that the local zeta integrals converge for $\Re(s) \ge 1$ (see Jacquet--Shalika~\cite[(3.17)]{JS}). We can easily reduce to the unitary case. Since $\pi$ is cuspidal, we know that $\pi = {^u}\pi \otimes |\det |^{\sigma}$ and this implies that $\tilde \pi = \widetilde {{^u} \pi} \otimes |\det |^{-\sigma}$, where ${^u} \pi$ and ${^u}\tilde \pi=\widetilde {{^u} \pi}$ are unitary representations, and $\sigma \in \R.$ 
Hence, $L(s, \pi \times \tilde \pi), \ D(s, W, \tilde W, \Phi) $ and $\Psi_v (s, W_v, \tilde W_v, \Phi_v)$ are all the same as their unitary versions. As seen below, we will be interested in the value of this function at $s=1$. We recall from Zhang~\cite[\S 3.1]{zhang} that
$$
\Psi_v ( 1, W_v, \tilde W_v, \Phi_v) \ = \ \hat \Phi_v (0) \cdot \varTheta_v (W_v, \tilde W_v), 
$$
where, $\varTheta_v$ is given by the integral
$$
\varTheta_v (W_v, \tilde W_v) \ = \ \int_{N_{n-1}(F_v) \backslash \GL_{n-1}(F_v) } W \begin{pmatrix} h & \\ & 1 \end{pmatrix} \tilde W \begin{pmatrix} h & \\ & 1 \end{pmatrix} \ dh.
$$
We remark here that $\varTheta_v$ is $\GL_n(F_v)$-equivariant, i.e., $\varTheta_v(gW_v, g\tilde W_v) = \varTheta_v (W_v, \tilde W_v)$, for all $g \in \GL_n (F_v)$.

Let $S= S_\pi \cup S_\infty,$ where $S_\pi$ is the finite set of finite places where $\pi$ is ramified and 
$S_\infty$ is the finite set of all archimedean places of $F$. For any $v \notin S$,  we know that 
$\Psi_v ( s, W_v, \tilde W_v, \Phi_v) = L_v(s, \pi_v \times \tilde \pi_v)$; see \cite{Cog}. Multiplying and dividing the right hand side of (\ref{eqn:eulerian}) by the $L$-factors for $v \in S$, we get
\begin{equation}
\label{D-L-relation}
D(s, \varphi, \tilde \varphi, \Phi) \ = \ 
L(s, \pi \times \tilde \pi) \cdot \prod_{v \in S} \frac{\Psi_v (s, W_v, \tilde W_v, \Phi_v)}{L_v (s, \pi_v \times \tilde \pi_v)}, 
\quad \Re(s) \gg 0. 
\end{equation}
All the terms on the right hand side admit an analytic continuation to a meromorphic function to all of $\C$, hence it makes sense to evaluate (or take residues) at $s = 1.$ It is well known that $L(s, \pi \times \tilde \pi)$ has a simple pole at $s=1$. We have the factorization
\begin{equation}
\label{eqn:L-1-Ad-defn}
L(s, \pi \times \tilde \pi) = \tilde\zeta_F (s) \cdot L(s, \Ad^0, \pi), 
\end{equation}
which defines $L(s, \Ad^0, \pi);$ here $\tilde\zeta_F(s)$ is the Dedekind zeta function of $F$. The simple pole at $s=1$ of 
$L(s, \pi \times \tilde \pi)$ comes from the simple pole at $s=1$ of $\tilde\zeta_F(s).$ Take residues at $s=1$ to get 
\begin{equation}
\label{L-residue}
\Res_{s=1} L(s, \pi \times \tilde \pi) \ = \ c_F \cdot L(1, \Ad^0, \pi)
\end{equation}
where $c_F = \Res_{s=1} \tilde\zeta_F(s)$; note, in particular, that $L(1, \Ad^0, \pi) \neq 0.$ 

For any $v\in S$, from the inputs going into the local functional equation, we know that the quotient 
$\Psi_v (s, W_v, \tilde W_v, \Phi_v)/L_v (s, \pi_v \times \tilde \pi_v)$ admits a continuation to an entire function. Define
\begin{equation}
\label{eqn:local-constants}
c_v(W_v,\tilde W_v, \Phi_v): = \left.  \frac{\Psi_v (s, W_v, \tilde W_v, \Phi_v)}{L_v (s, \pi_v \times \tilde \pi_v)} \right|_{s=1} = \frac{\hat \Phi_v (0) \varTheta_v (W_v, \tilde W_v)}{L_v (1, \pi_v \times \tilde \pi_v)}, 
\end{equation}
and 
\begin{equation}
\label{eqn:new-local-constants}
c^\#_v(W_v,\tilde W_v): = \frac{\varTheta_v (W_v, \tilde W_v)}{L_v (1, \pi_v \times \tilde \pi_v)}. 
\end{equation}
It is well known that if ${^u}\pi_v$ is a unitary, generic, irreducible representation of $\GL_n(F_v)$ then 
$L(1, {^u}\pi_v \times {^u}\tilde\pi_v)$ is finite; hence for us, $L_v (1, \pi_v \times \tilde \pi_v)$ is finite. 

Taking residues at $s=1$ on both sides of (\ref{D-L-relation}) and using (\ref{D-residue}), (\ref{L-residue}) and 
(\ref{eqn:local-constants}), we get the following integral representation of $L(1, \Ad^0, \pi)$: 
\begin{equation}
\label{eqn:integral-repn-L-1-ad}
c_F  \cdot  \left(\prod_{v \in S} \fr c_v (W_v, \tilde W_v, \Phi_v) \right)  \cdot L(1, \Ad^0, \pi) \ = \ 
\frac{\mathrm{vol}(F^\times \backslash \A^1_F) \cdot \hat \Phi_f (0) \cdot \hat \Phi_\infty (0)}{n}  \cdot \langle \varphi, \tilde \varphi \rangle.
\end{equation}
We will see in \S\,\ref{sec:ramified} below that, for $v \in S_\pi$ and for our special choice of local Whittaker vectors, one has  $\fr c_v (W_v, \tilde W_v, \Phi_v) \neq 0$ . Define
\begin{equation}
\label{eqn:omega-ram}
\fr p_\ram (\pi) \ := \  \left(\prod_{v \in S_\pi} \fr c_v(W_v, \tilde W_v, \Phi_v)\right)^{-1}.
\end{equation}
Define: 
\begin{equation}
\label{eqn:omega-F}
\omega_F \ := \  \frac{\mathrm{vol}(F^\times \backslash \A^1_F) \cdot \hat \Phi_f(0)}{n\cdot c_F} 
\end{equation}
Since $\hat \Phi_\infty (0) \not = 0$, the integral representation for $L(1, \Ad^0, \pi)$ can be re-written as: 
\begin{equation}
\label{eqn:integral-repn-L-1-ad-again}
\frac{\prod_{v \in S_\infty} \fr c^\#_v (W_v, \tilde W_v) \cdot L(1, \Ad^0, \pi)}{\omega_F \cdot \fr p_\ram (\pi)} 
\ = \ 
\langle \varphi, \tilde \varphi \rangle, 
\end{equation}
for the special choice of cusp forms and Schwartz--Bruhat functions made in (\ref{eqn:special-choice-vectors}). 
The main algebraicity theorem that we prove (see Thm.\,\ref{thm:l-value} below) involves giving a cohomological interpretation to this integral representation of $L(1, {\rm Ad}^0, \pi).$

\medskip
\subsection{Ramified calculations}
\label{sec:ramified}

In this section, we study the quantities $\fr c_v (W_v, \tilde{W_v}, \Phi_v)$ for $v \in S_\pi$ and justify the definition of 
$\fr p_\ram(\pi)$ in (\ref{eqn:omega-ram}). Throughout this section, $\Phi_v$ will be the characteristic function of $\calo_v^n,$ and will often be suppressed from the notation. This is a purely local statement, and furthermore recall that a local component of a global cuspidal representation of 
$\GL_n$ is generic, i.e., admits a local Whittaker model.  For most of this subsection, we suppress the subscript $v$ and 
let $\pi$ be an irreducible admissible generic representation of $\GL_n(F)$ for a non-archimedean local field $F$. 
The essential vector in a Whittaker model of $\pi$ will be denoted $W(\pi).$ 
For later use in studying algebraicity properties, we will also study the behavior of $\fr c (W(\pi), W(\tilde\pi))$
under any $\sigma \in \Aut(\C).$ Given $\pi$, for the definition and some basic properties of the 
conjugated representation ${}^\sigma\pi$, the reader is referred to Clozel~\cite{clozel} or Waldspurger~\cite{waldy}; 
see also the discussion in Grobner--Raghuram~\cite[Sect.\,7.1]{grobner-raghuram-ijnt}. 
We omit the proofs of the following two lemmas which may be proved along the lines of the proof of 
\cite[Prop.\,3.17]{raghuram-imrn}. 

\begin{lemma}
\label{1st-lemma}
(We have currently adopted local notations.) Let $\pi$ be an irreducible admissible generic representation of $\GL_n(F).$ Then 
 $$
 \sigma \left(L(a, \pi \times \tilde{\pi}) \right) \ = \ 
 L(a, {^\sigma} \pi \times {^\sigma} \tilde{\pi}),\quad \forall a \in \Z.
$$
\end{lemma}

\medskip

\begin{lemma}
\label{2nd-lemma}
(We have currently adopted local notations.) Suppose $\pi$ is an irreducible, admissible, generic representation of 
$\GL_n(F)$, and $L(s, \pi) = \prod_{i = 1}^l (1 - \alpha_i q^{-s})^{-1}$. Define $\pi^{\rm un}$ to be the 
irreducible spherical representation of $\GL_l(F)$ with Satake parameters $\alpha_1, \dots, \alpha_l.$ 
For any $\sigma \in \Aut(\C)$ we have
$$
{}^\sigma (\pi^\mathrm{un}) = ({^\sigma} \pi)^\mathrm{un}\cdot \epsilon_\sigma^{n-l}, 
$$
where $\epsilon_\sigma : F^\times \to \C^\times$ is the quadratic character defined as 
$\epsilon_\sigma(x) = |x|^{1/2}/\sigma(|x|^{1/2})$
\end{lemma}

Consider the numerator of $\fr c(W(\pi), W(\tilde{\pi}))$ as in (\ref{eqn:local-constants}).  Greni\'e \cite[p.\,306]{Gr} showed that
$$
\Psi(1, W(\pi), W(\tilde{\pi}), \Phi) \ = \ L(j, \pi^{\mathrm{un}} \times \tilde \pi^{\mathrm{un}}), \quad \text{ where } j=n+1-l
$$
for the $\Phi$ fixed at the beginning of this section.
In particular, $\Psi(1, W(\pi), W(\tilde{\pi}), \Phi) \neq 0$. Furthermore, as mentioned above, we know 
that $L(1, \pi \times \tilde\pi)$ is finite (i.e., not a pole). Hence, 
$\fr c(W(\pi), W(\tilde{\pi})) \neq 0,$ justifying the definition in (\ref{eqn:omega-ram}). 

\smallskip

For rationality properties apply $\sigma \in \Aut (\C)$ to see: 
\begin{equation}
\label{eqn:local-constant-rational}
\begin{split}
\sigma (\fr c (W (\pi), W(\tilde{\pi}))) 
& \ = \ \frac{\sigma (\Psi (1, W(\pi), W(\tilde{\pi}), \Phi))}{\sigma(L(1, \pi \times \tilde{\pi}))} 
\ = \ \frac{\sigma(L(j, \pi^{\mathrm{un}} \times \tilde \pi^{\mathrm{un}}))} {\sigma(L(1, \pi \times \tilde{\pi}))} \\
&\ = \ \frac{L(j,{^\sigma} (\pi^{\mathrm{un}}) \times {^\sigma}(\tilde \pi^{\mathrm{un}}))}{L(1, {^\sigma} \pi \times {^\sigma} \tilde{\pi})} \quad \text{[by Lemma \ref{1st-lemma}]}  \\
& \ = \ \frac{L(j,({^\sigma}\pi)^{\mathrm{un}} \times ({^\sigma}\tilde \pi)^{\mathrm{un}})}{L(1, {^\sigma} \pi \times {^\sigma} \tilde{\pi})} \quad \text{[by Lemma \ref{2nd-lemma}]}  \\
& \ = \ \frac{\Psi (1, W({^\sigma} \pi), W({^\sigma} \tilde{\pi}), \Phi)}{L(1, {^\sigma} \pi \times {^\sigma} \tilde{\pi})} \\
& \ = \ \fr c (W ({^\sigma} \pi), W({^\sigma} \tilde{\pi})).
\end{split}
\end{equation}

\smallskip

Going back to the global situation, suppose $\pi$ is a cuspidal automorphic representation of cohomological type, then it will make sense to consider ${}^\sigma\!\pi$, which would have the property that 
$({}^\sigma\!\pi)_v = {}^\sigma\!(\pi_v);$ see Clozel~\cite[Thm.\,3.13]{clozel}.  From 
(\ref{eqn:local-constant-rational}) applied to every $v \in S_\pi$, we conclude
\begin{equation}
\label{eqn:omega-ran-rational}
\sigma (\fr p_\ram (\pi)) \  = \ \fr p_\ram ({}^\sigma\!\pi).
\end{equation}
We also remark that this last equation can also be obtained using the relationship between the local zeta integrals and the integrals $\varTheta$, as done in \cite{GHL}.

\bigskip
\section{\bf Whittaker models and automorphic cohomology}
\label{whitt-models}

The main purpose of this section is to give a cohomological interpretation to the integral representation 
in (\ref{eqn:integral-repn-L-1-ad-again}) for $L(1, {\rm Ad}^0, \pi).$

\subsection{The basic set-up to study the cohomology of arithmetic groups}
\label{sec:basic-setup}
We briefly review the basic set-up and for all the details we refer the reader to 
\cite[\S\,1]{harder-raghuram-cras} and \cite[\S\,2.3]{raghuram-2013}. 
For any open compact subgroup $K_f$ of $G(\A_f)$, consider the space
$$
S_{K_f}^G = G(\Q) \backslash G(\A)/K_\infty^\circ K_f, 
$$ 
where $K_\infty = C_\infty S (\R)$ and $C_\infty$ is a maximal compact subgroup of $G_\infty = G(\R)$.
The connected component of the identity in $K_\infty$, denoted $K_\infty^\circ$, is isomorphic to 
$({\rm SO}_n(\R)^{r_1} \times {\rm U}_n(\C)^{r_2}) \cdot \R_{>0},$ where $r_1$ and $r_2$ are the number of real and complex embeddings of $F$ and $r_1+2r_2 = d$. We identify $K_\infty/K_\infty^\circ =: \pi_0(K_\infty) = \pi_0(G(\R)).$ 

Let $E$ be as in the introduction, which will be viewed as a subfield of $\C$ via $\iota.$ 
Let $\lambda \in X^+_{00}(T)$ be a strongly pure dominant-integral weight; see \cite[Def.\,2.5]{raghuram-2013}. Let 
$\calm_{\lambda, E}$ be the algebraic irreducible representation of $G$ over the field $E$ with highest weight 
$\lambda.$ Let $\tilde\lambda$ be the `dual' weight which is a dominant integral weight that is 
the highest weight of the contragredient of 
$\calm_{\lambda, E}.$ Let $\tcalm_{\lambda, E}$ denote the associated sheaf on $S_{K_f}^G.$ We are interested in the cohomology groups 
$H^{\sbt}(S_{K_f}^G, \tcalm_{\lambda, E}).$ The image of cohomology with compact supports in full cohomology is called inner or interior cohomology: 
$
H^{\sbt}_!(S_{K_f}^G, \tcalm_{\lambda, E}) \ := \ 
{\rm Image}\left(H^{\sbt}_c(S_{K_f}^G, \tcalm_{\lambda, E}) \ \to \ H^{\sbt}(S_{K_f}^G, \tcalm_{\lambda, E})\right). 
$
There is a Hecke action, i.e., an action of $\HGK \times \pi_0(G(\R))$ on $H^{\sbt}(S_{K_f}^G, \tcalm_{\lambda, E})$ which stabilizes $H^{\sbt}_!$, and furthermore, inner cohomology is semi-simple for this Hecke action. For all these assertions and with more details, see Harder--Raghuram~\cite[\S 2]{harder-raghuram-preprint}.

\smallskip

Passing to a transcendental level via $\iota$, we also have cuspidal cohomology: 
$$
H^{\sbt}_{\rm cusp}(S_{K_f}^G, \tcalm_{\lambda, \C}) \ \subset \ 
H^{\sbt}_!(S_{K_f}^G, \tcalm_{\lambda, \C}) \ \subset \ 
H^{\sbt}(S_{K_f}^G, \tcalm_{\lambda, \C}).
$$
The definition of cuspidal cohomology is via relative Lie algebra cohomology, in terms of which, 
as $\pi_0(G_n(\R)) \times \HGK$-modules, we have 
\begin{equation}
\label{eqn:cusp-coh}
H^{\sbt}_{\rm cusp}(S^G_{K_f}, \tcalm_{\lambda,\C}) \ = \ 
\bigoplus_{\pi} H^{\sbt}(\g, K_{\infty}^\circ;  \pi_{\infty} \otimes  \calm_{\lambda,\C}) \otimes \pi_f^{K_f}. 
\end{equation}
We say that $\pi$ contributes to the cuspidal cohomology of $G$ with coefficients in $\calm_{\lambda,\C}$, and 
we write $\pi \in {\rm Coh}(G, \lambda, K_f)$,  
if $\pi$ has a nonzero contribution to the above decomposition; equivalently, if $\pi_{\infty}$ after twisting by 
$\calm_{\lambda,\C}$ has nontrivial relative Lie algebra cohomology and $\pi_f$ has $K_f$-fixed vectors. 
If $\pi \in {\rm Coh}(G, \lambda, K_f)$, then via well-known results of Clozel~\cite{clozel}, we have explicit knowledge of 
$\pi_\infty$ 
(see \cite[Sect.\,2.4]{raghuram-2013}), from which one can compute the relative Lie algebra cohomology groups.  
Define the numbers $b_n^F$ and $\tilde{t}_n^F$ as in \cite[Prop.\,2.14]{raghuram-2013}. 
Then we have 
$$
H^{\sbt}_{\rm cusp}(S^G_{K_f}, \tcalm_{\lambda,\C}) \neq 0 \ \iff \ 
b_n^F \leq \sbt \ \leq \tilde{t}_n^F. 
$$
For brevity, let $b=b_n^F$ and $t= \tilde{t}_n^F$ denote the lower and upper end of the ``cuspidal range." 
Note that $b + t = \dim(\SGK)$.

\medskip

\subsection{Betti--Whittaker periods}
\label{sec:betti-whittaker}

\subsubsection{\bf Comparing Whittaker models and cohomological models}
\label{sec:whit-coh}
Let $\pi \in {\rm Coh}(G, \lambda, K_f)$ as above, and we focus only on the extreme degrees 
$\sbt \in \{b, t\}.$ If $n$ is even then every character $\varepsilon = (\varepsilon_v)_{v \in S_r}$ of $\pi_0(K_{\infty})$ 
appears  once in $H^{\sbt}(\g, K_{\infty}^0;  \pi_{\infty} \otimes  \calm_{\lambda,\C});$ hence,   
$$
H^{\sbt}_{\rm cusp}(S^{G}_{K_f}, \tcalm_{\lambda, \C}) \ = 
\bigoplus_{\pi  \in {\rm Coh}(G, \lambda, K_f)}
\bigoplus_{\varepsilon \,  \in \, \widehat{\pi_0(K_{\infty})}} 
\varepsilon \otimes \pi_f^{K_f}, \quad \mbox{(when $n$ is even)}. 
$$
Note that each $\pi_f^{K_f}$ appears $2^{r_1}$ times in cuspidal cohomology in degree $b$ or $t$.  
However, if $n$ is odd then $H^{\sbt}(\g, K_{\infty}^0;  \pi_{\infty} \otimes  \calm_{\lambda,\C})$ is one-dimensional and the character of $\pi_0(K_{\infty})$ which appears is denoted as $\varepsilon_{\pi_\infty}$. In this case, cuspidal cohomology decomposes as 
$$
H^{\sbt}_{\rm cusp}(S^{G}_{K_f}, \tcalm_{\lambda, \C}) = \bigoplus_{\pi  \in {\rm Coh}(G, \lambda, K_f)}
\varepsilon_{\pi_\infty} \otimes \pi_f^{K_f}, \quad \mbox{(when $n$ is odd)}.
$$
Following \cite{RS}, an $r_1$-tuple of signs $\varepsilon = (\varepsilon_v)_{v\in S_r}$ is said to be 
{\it permissible} for $\lambda$,  if $\varepsilon = \varepsilon_{\pi_\infty}$ for some 
$\pi \in {\rm Coh}(G, \lambda, K_f)$ when $n$ is odd, and is any of the possible $2^{r_1}$ signatures when $n$ is even. 
For any such permissible signature,  we then have the following series of isomorphisms:
\begin{equation}\label{big-iso}
\begin{split}
\calw(\pi_f)^{K_f} &\longrightarrow 
\calw(\pi_f)^{K_f} \otimes H^{\sbt} (\fr g, K^\circ_\infty; \calw(\pi_\infty) \otimes \calm_{\lambda, \C})(\varepsilon) 
\longrightarrow 
H^{\sbt} (\fr g, K^\circ_\infty; \calw(\pi)^{K_f} \otimes \calm_{\lambda, \C})(\varepsilon) \\
&\longrightarrow 
H^{\sbt} (\fr g, K^\circ_\infty; V_\pi^{K_f} \otimes \calm_{\lambda, \C})(\varepsilon)  
\longrightarrow 
H^{\sbt} (\fr g, K^\circ_\infty;  \cala_0^\infty (\omega)^{K_f} \otimes \calm_{\lambda, \C})(\pi_f \times \varepsilon) \\
&\longrightarrow 
H_!^{\sbt} (\SGK, \tcalm_{\lambda,\C})(\pi_f \times \varepsilon). 
\end{split}
\end{equation}
We denote by $\calf_{\sbt, \varepsilon} (\pi_f)$ the isomorphism obtained by composing all these maps.

\smallskip

\subsubsection{\bf Fixing a basis at infinity}
Note that the first map in the isomorphism above is given by fixing a basis element $[\pi_{\infty}]^\varepsilon$ for the 
one-dimensional vector space $H^{\sbt} (\fr g, K^\circ_\infty; \calw(\pi_\infty) \otimes \calm_{\lambda,\C})(\varepsilon)$. We describe in detail how to do this for the bottom cohomology degree. By the K\"unneth theorem, we have
$$
H^{b} (\fr g, K^\circ_\infty; \calw(\pi_\infty) \otimes \calm_{\lambda,\C})(\varepsilon) \ = \ 
\bigotimes_{v \in S_\infty} H^{b_v} (\fr g_v, K^\circ_v; \calw(\pi_v) \otimes \calm_{\lambda_v, \C})(\varepsilon_v).
$$
Take $[\pi_{\infty}]^\varepsilon = \otimes [\pi_v]^{\varepsilon_v}$, where $[\pi_v]^{\varepsilon_v}$ is a basis for the 
$1$-dimensional space
\begin{equation*}
\begin{split}
H^{b_v} (\fr g_v, K^\circ_v; \calw(\pi_v) \otimes \calm_{\lambda_v, \C})(\varepsilon_v) 
& \ = \ \Hom_{K_v^\circ} (\wedge^{b_v} \fr g_v/\fr k_v , \calw (\pi_v) \otimes \calm_{\lambda_v, \C} )(\varepsilon_v) \\
&\ = \ \left[ \wedge^{b_v} (\fr g_v/\fr k_v)^\sv \otimes \calw (\pi_v) \otimes \calm_{\lambda_v, \C} \right]^{K_v^\circ}(\varepsilon_v).
\end{split}
\end{equation*}

If $\{X_{k,v} \}$ is a basis for $\fr g_v/\fr k_v$ and 
$\{X_{k,v}^\sv \}$ is the dual basis for $(\fr g_v/\fr k_v)^\sv$,  and if $\{m_{\alpha_v}\}$ is a basis for 
$\calm_{\lambda_v,\C}$ then we can write
$$
[\pi_v]^{\varepsilon_v} = \sum_{\underline{i_v}, \alpha_v} X^\sv_{\underline{i_v}} \otimes W_{\underline{i_v}, \alpha_v, \varepsilon_v, v} \otimes m_{\alpha_v}, 
$$
where $X^\sv_{\underline{i_v}} = X^\sv_{i_1, v} \wedge X^\sv_{i_2,v} \wedge \cdots  \wedge X^\sv_{i_{b_v},v}$ for 
$\underline{i_v} = (i_1, i_2, \cdots, i_{b_v} ).$ 

Similarly, consider the contragredient representation 
$\tilde \pi$ and suppose $\eta$ is a permissible signature for $\tilde \pi.$ We fix a generator for the top-degree cohomology for $\tilde \pi$: 
$$
[\tilde \pi_v]^{\eta_v} = \sum_{\underline{j_v}, \beta_v} X^\sv_{\underline{j_v}} \otimes \tilde W_{\underline{j_v}, \beta_v, \eta_v, v} \otimes \tilde m_{\beta_v}. 
$$

\smallskip

\subsubsection{\bf The Betti--Whittaker periods}
If $\pi \in {\rm Coh}(G, \lambda, K_f)$ and $\sigma \in \Aut(\C)$ then by a well-known result of 
Clozel (\cite[Thm.\,3.13]{clozel}) we have 
${}^\sigma\pi \in {\rm Coh}(G, {}^\sigma\lambda, K_f)$.  
Following \cite[Def./Prop.\,3.3]{RS}, the Betti--Whittaker periods $\fr p^\epsilon (\pi) \in \C^\times$  can be chosen in such a way that the following diagram commutes:
$$
\xymatrix@1{
\calw(\pi_f) \ar[rrr]^-{\calf^\circ_{b,\varepsilon} (\pi)} \ar[d]^-{\sigma} &&& 
H_!^b (\SGK, \tcalm_{\lambda, \C} )(\pi_f \times \varepsilon) \ar[d]^-{\sigma} \\
\calw(\leftidx{^\sigma} \pi_f) \ar[rrr]^-{\calf^\circ_{b, \varepsilon} (\leftidx{^\sigma} \pi_f)} 
&&& H_!^b (\SGK, \tcalm_{{^\sigma} \lambda, \C} )(\leftidx{^\sigma} \pi_f \times \varepsilon)
}$$
where $\calf^\circ_{b,\varepsilon} (\pi) = \calf_{b,\varepsilon} (\pi)/\fr p^\varepsilon (\pi)$ and 
$\calf^\circ_{b, \varepsilon} (\leftidx{^\sigma} \pi_f) = 
\calf_{b, \varepsilon} (\leftidx{^\sigma} \pi_f)/\fr p^{\varepsilon} (\leftidx{^\sigma} \pi)$, which preserve rational structures on either side. The rationality field $\Q(\pi)$ is defined as in 
\cite[\S\,7.1]{grobner-raghuram-ijnt}. 
The collection $\{p^{\varepsilon}({}^\sigma\pi) : \sigma \in \Aut(\C) \}$ is well-defined in 
$(\Q(\pi) \otimes \C)^\times/\Q(\pi)^\times.$
Similarly, we have the Betti--Whittaker periods $\fr q^\eta(\tilde \pi)$ for $\tilde \pi$ and $\eta$ by considering cuspidal cohomology in top-degree. (Let's note here that in \cite{RS}, the right vertical arrow mapped the 
$\pi_f \times \varepsilon$-isotypic component to the ${}^\sigma\!\pi_f \times {}^\sigma\!\varepsilon$-isotypic component; this is a small mistake 
as the map induced by $\sigma$ in cohomology maps the $\pi_f \times \varepsilon$-isotypic component to the 
${}^\sigma\!\pi_f \times \varepsilon$-isotypic component; see \cite{raghuram-2013} for more explanations on this 
issue.)

\smallskip

\subsubsection{\bf The rational cohomology classes}
Recall that for any finite place $v$, $W_v = W_{\pi_v}$ stands for the essential vector of $\pi_v$ in its $\psi_v$-Whittaker model
$\calw(\pi_v, \psi_v).$ Let $W_f = \otimes W_v \in \calw(\pi_f).$ 
Similarly, $\tilde W_v = W_{\tilde \pi_v}$ in $\calw(\tilde \pi_v, \psi_v^{-1})$ is the essential vector, 
and $\tilde W_f  = \otimes \tilde W_v \in  \calw(\tilde \pi_f).$ It follows from \cite[Sect.\,3.1.3]{raghuram-imrn} and 
\cite[Prop.\,3.21]{raghuram-imrn} that for any $\sigma \in \Aut(\C)$ we have 
${}^\sigma W_{\pi_v} = W_{{}^\sigma\pi_v}.$ 
Define: 
$$
\vartheta_{b, \varepsilon} = \calf_{b, \varepsilon} (W_f) \quad {\rm and} \quad 
\tilde \vartheta_{t, \eta} = \calf_{t,\eta} (\tilde W_f), 
$$
where it is understood that $\varepsilon$ and $\eta$ are permissible signatures. It follows from the definition of the periods, and the above rationality properties of $W_f$ and $\tilde W_f$ that the normalized classes: 
$$
\vartheta^\circ_{b,\varepsilon} \ := \ \vartheta_{b, \varepsilon}/\fr p^\varepsilon (\pi)  \quad {\rm and} \quad
\tilde \vartheta^\circ_{t,\eta} \ := \ \tilde \vartheta_{t, \eta}/\fr q^\eta (\tilde \pi) 
$$
are rational, i.e., $\vartheta^\circ_{b,\varepsilon} \in H^b_! (\SGK, \tcalm_{\lambda, E})$ and 
$\tilde \vartheta^\circ_{t,\eta} \in H^t_! (\SGK, \tcalm_{\tilde \lambda, E})$ for $E$ as before. 
Note that $E$ being large enough has the effect that 
the Hecke-summand $\pi_f^{K_f} \times \varepsilon$  splits off in inner cohomology, and similarly for 
$\tilde \pi^{K_f} \times \eta$.

\smallskip

\subsubsection{\bf Integral cohomology and refined periods}
\label{sec:int-refined-periods}
We now define an integral structure on these cohomology groups through a subsheaf 
$\calm_{\lambda, \calo} \subset \calm_{\lambda, E}\subset \calm_{\lambda, \C} = \calm_\lambda$. 
For a precise definition of these integral sheaves, see \cite[\S 2]{Ha}. These inclusions induce the following maps on inner cohomology: 
$$
H^b_!(\SGK, \tcalm_{\lambda, \calo}) \ \to \ 
H^b_! (\SGK, \tcalm_{\lambda, E}) \ \to \ 
H^b_! ( \SGK, \tcalm_{\lambda, \C}) \ = \ 
H^b_! (\SGK, \tcalm_{\lambda, E}) \otimes \C.
$$ 
The first map here is in general not injective and we define our integral cohomology
$$
\bar H^b_! (\SGK, \tcalm_{\lambda, \calo}) \ = \ \text{Image} \left(H^b_!(\SGK, \tcalm_{\lambda, \calo}) \to 
H^b_! (\SGK, \tcalm_{\lambda, E}) \right).
$$
Define $\bar H^t_!(\SGK, \tcalm_{\tilde \lambda, \calo})$ in a similar manner. These integral cohomology groups are 
$\calo$-lattices of full rank in the respective rational cohomology groups. They are stable under the action of 
$\HGK \times \pi_0(G(\R))$ for a suitably renormalized action of Hecke operators (see \S\ref{proof-thmb} below), and 
hence we can define isotypic components:  
$$
\bar H^b_!(\SGK, \tcalm_{\lambda, \calo}) (\pi_f^{K_f} \times \varepsilon)  \quad {\rm and} \quad 
\bar H^t_!(\SGK, \tcalm_{\tilde \lambda, \calo}) (\tilde \pi_f^{K_f} \times \eta).
$$
Thus, we have
$$
\bar H^b_!(\SGK, \tcalm_{\lambda, \calo}) (\pi_f^{K_f} \times \varepsilon) 
\ \subset \ 
H^b_! (\SGK, \tcalm_{\lambda, E})(\pi_f^{K_f} \times \varepsilon)  \ \ni \ \vartheta^\circ_{b, \varepsilon} (\pi), 
$$
and, similarly, 
$$
\bar H^t_!(\SGK, \tcalm_{\tilde \lambda, \calo}) (\tilde \pi_f^{K_f} \times \eta) 
\ \subset \ 
H^t_! (\SGK, \tcalm_{\tilde \lambda, E})(\tilde \pi_f^{K_f} \times \eta) \ \ni \ \tilde \vartheta^\circ_{t, \eta} (\tilde \pi).
$$

So far $K_f$ could have been any open compact-subgroup of $G(\A_f)$ which is deep enough so that the vector 
$W_f$ built out of essential vectors is $K_f$-fixed. However, we now choose $K_f$ optimally so that 
$\pi_f^{K_f}$ is a one-dimensional space; this is possible via the main results of \cite{j-ps-s} about conductors of representations. In particular, 
$H^b_! (\SGK, \tcalm_{\lambda, E})(\pi_f^{K_f} \times \varepsilon)$ is a one-dimensional $E$-vector space, and 
$H^b_!(\SGK, \tcalm_{\lambda, \calo}) (\pi_f^{K_f} \times \varepsilon)$ is a free rank-one $\calo$-submodule. 
We refine the definition of our periods $\fr p^\varepsilon (\pi)$ and $\fr q^\eta (\tilde \pi)$ such that
\begin{equation}
\label{eqn:refined-periods}
\begin{split}
\vartheta_{b, \varepsilon}/\fr p^\varepsilon (\pi) &
\ = \ \vartheta^\circ_{b,\varepsilon}  \ \in \ \bar H^b_!(\SGK, \tcalm_{\lambda, \calo}) (\pi_f \times \varepsilon), \\ 
& \\ 
\tilde \vartheta_{t, \eta}/\fr q^\eta (\tilde \pi) &
\ = \ \tilde \vartheta^\circ_{t,\eta} \ \in \   
\bar H^t_!(\SGK, \tcalm_{\tilde \lambda, \calo}) (\tilde \pi_f \times \eta), 
\end{split} 
\end{equation}
and that $\vartheta^\circ_{b,\varepsilon}$ and $\tilde \vartheta^\circ_{t,\eta}$ are generators of these rank one $\calo$-modules.

\medskip
\subsection{Cohomological pairing and the main theorem on adjoint $L$-values}

\subsubsection{\bf The main idea behind the rationality result}
\label{sec:main-idea}
Consider the Whittaker spaces of $\pi_f$ and $\tilde \pi_f$. Tacking on certain `cohomological vectors,'  we can go into the space of cusp forms $V_\pi$ and $V_{\tilde \pi}$ and then take the Petersson inner product. On the other hand,  
we can map into the bottom and top degree inner cohomology groups and consider the Poincar\'e duality pairing:
\begin{equation*}
\xymatrix@C=1pt{
V_\pi^{K_f} \cong \calw(\pi_f)^{K_f} \otimes \calw(\pi_\infty) & \bigotimes & V_{\tilde \pi}^{K_f} \cong 
\calw(\tilde \pi_f)^{K_f} \otimes \calw(\tilde \pi_\infty) \ar[rrrr] & & & & \C\\
\calw (\pi_f)^{K_f} \ar[d]^{\calf_{b,\varepsilon} (\pi_f)} \ar[u] & & \calw(\tilde \pi_f)^{K_f}  \ar[d]^{\calf_{t,\eta} (\tilde \pi_f)} \ar[u] &  \\
{ H_!^{b} (\SGK, \tcalm_{\lambda,\C})(\pi_f \times \varepsilon)} &\bigotimes & 
{ H_!^{t} (\SGK, \tcalm_{\tilde \lambda,\C})(\tilde \pi_f \times \eta)} \ar[rrrr] & \quad & & & \C.\\
}
\end{equation*}
Here the top row is the Petersson inner product and the bottom row is Poincar\'e duality pairing on cohomology. 
The vertical maps down to the cohomology groups are isomorphisms and they depend on the choices, 
$[\pi_\infty]^\varepsilon$ and $[\tilde \pi_\infty]^{\eta}.$ The Poincar\'e pairing can be computed after translating into relative Lie algebra cohomology to relate it to the pairing via Petersson inner product. On the other hand, wedge product at the level of differential forms is the same as cup product, i.e., the Poincar\'e pairing has nice rationality properties.

\smallskip
\subsubsection{\bf A cohomological pairing}

For any $\underline{\bf i} = (\underline{i_v}), \ \underline \alpha = (\alpha_v), \ 
\underline{\bf j} = (\underline{j_v}), \text{ and } \underline \beta = (\beta_v) $ we let $\varphi_{\underline{\bf i},\underline{\alpha}, \varepsilon}\in V_\pi$ and $\tilde \varphi_{\underline{\bf j}, \underline \beta, \eta} \in V_{\tilde \pi}$ be cuspidal automorphic forms corresponding to the Whittaker functions $W_f \bigotimes \otimes_v W_{\underline{i_v}, \alpha_v, \varepsilon_v, v}$ and $\tilde W_f \bigotimes \otimes_v \tilde W_{\underline{j_v}, \beta_v, \eta_v, v},$
 respectively.  Then
$$
\langle \vartheta_{b, \varepsilon}, \tilde \vartheta_{t, \eta} \rangle = \sum_{\underline{\bf i}, \underline{\bf j}} \sum_{\underline \alpha, \underline \beta} s(\underline{\bf i}, \underline{\bf j}) \langle m_{\alpha_v}, \tilde m_{\beta_v} \rangle \langle \varphi_{\underline{\bf i},\underline{\alpha}, \varepsilon},  \tilde \varphi_{\underline{\bf j}, \underline \beta, \eta} \rangle, 
$$
where $\langle m_{\alpha_v}, \tilde m_{\beta_v} \rangle$ comes from a canonical pairing 
$\calm_{\lambda,E} \times \calm_{\tilde \lambda,E} \to E$ which has been fixed once and for all, and 
$s(\underline{\bf i}, \underline{\bf j}) = \prod s(\underline{i_v}, \underline{j_v})$ where $s(\underline{i_v}, \underline{j_v})$ is defined by the formula
$$X^\sv_{i_1, v} \wedge \cdots \wedge X^\sv_{i_{b_v}, v} \wedge X^\sv_{j_1, v} \wedge \cdots \wedge X^\sv_{i_{t_v}, v} = s(\underline{i_v}, \underline{j_v}) (X^\sv_{1, v} \wedge \cdots \wedge X^\sv_{b_v + t_v, v}).$$ Using the relationship with the special value of the adjoint $L$-function from the previous section, we get
\begin{equation}
\label{eqn:l-value-pairing}
\langle \vartheta_{b, \varepsilon}, \tilde \vartheta_{t, \eta} \rangle = \frac{\fr c_\infty (\pi, \tilde \pi, \varepsilon, \eta) \cdot L(1, \Ad^0, \pi)}{\omega_F \cdot \fr p_\ram (\pi)}
\end{equation}
where 
$$\fr c_\infty (\pi, \tilde \pi, \varepsilon, \eta) = \prod_{v \in S_\infty}  \sum_{\underline{i_v}, \underline{j_v}} \sum_{\alpha_v, \beta_v} s(\underline{i_v}, \underline{j_v}) \cdot \langle m_{\alpha_v}, \tilde m_{\beta_v} \rangle \cdot \fr c^\#_v (W_{\underline{i_v}, \alpha_v, \varepsilon_v, v}, \tilde W_{\underline{j_v}, \beta_v, \eta_v, v} ).$$

By a suitably refined version of Poincare duality (see, for example, Harder~\cite[Thm.\,4.8.9]{Ha-AlGeo1}), we have 
\begin{equation}
\label{eqn:l-value-integral-pairing}
\frac{ \langle \vartheta_{b, \varepsilon}, \tilde \vartheta_{t, \eta} \rangle}{\fr p^\varepsilon (\pi) \cdot \fr q^{\eta} (\tilde \pi)} 
\ = \  
\langle \vartheta^\circ_{b,\varepsilon}, \tilde \vartheta_{t, \eta}^\circ \rangle \in \calo.
\end{equation}

\smallskip

\subsubsection{\bf Compatibility between the permissible signatures $\varepsilon$ and $\eta$}\label{compatible-signatures}

\begin{proposition}
If $\eta \not = (-1)^{n-1} \varepsilon$, then $\langle \vartheta_{b, \varepsilon}, \tilde \vartheta_{t, \eta} \rangle = 0.$
\end{proposition}
\begin{proof}
At the level of relative Lie algebra cohomology at infinity the $\varepsilon$-isotypic component can pair non-trivially with $\eta$-isotypic component only if $\eta = (-1)^{n-1} \varepsilon$. The proposition follows immediately from this claim.

Since the relevant spaces at infinity are tensor products over the infinite places of the number field $F$, we fix without loss of generality a real place $v$ of $F$. The pairing we are interested in is given by
 \begin{equation*}
 \xymatrix@C=1pt{
 H^{b_v} (\fr g_v, K^\circ_v; \calw(\pi_v) \otimes M_{\lambda_v}) \ar@{=}[d] & \otimes & H^{t_v} (\fr g_v, K^\circ_v; \calw(\tilde \pi_v) \otimes M_{\tilde \lambda_v}) \ar[rrrr] \ar@{=}[d] & & & & \C \\
  \left[ \wedge^{b_v} (\fr g_v/\fr k_v)^\sv \otimes \calw (\pi_v) \otimes M_{\lambda_v} \right]^{K_v^\circ} & \otimes &  \left[ \wedge^{t_v} (\fr g_v/\fr k_v)^\sv \otimes \calw (\tilde \pi_v) \otimes M_{\tilde \lambda_v} \right]^{K_v^\circ} 
  \ar[rrrr] & & & & \C.\\
 }
 \end{equation*}
 Let $\delta_n = \delta_{n,v} = \text{Diag}(-1, 1, \dots, 1) \in K_v/K_v^\circ$.  For each $1 \le i < j \le n$, let $X_{ij}$ be the matrix with $1$ in the $ij$-th and $ji$-th entry and $0$ elsewhere. For each $1\le r \le n-1$, let $X_r$ be the matrix with $1$ in the $rr$-th entry, $-1$ in the $(r+1)(r+1)$-th entry, and $0$ elsewhere. This collection of vectors form a basis for $\fr g_v/\fr k_v$. The action of $K_v/K_v^\circ$ on $\fr g_v/\fr k_v$ is by conjugation and we see easily that 
 $\delta_n \cdot X_r = X_r$ for all $r$, 
 $\delta_n \cdot X_{ij} = X_{ij}$ for all $1 < i<j \le n,$ 
 and $\delta_n \cdot X_{1j} = -X_{1j}$ for all $1 <j \le n$. Using this one can verify that 
 $$\langle \delta_n \cdot a, \delta_n \cdot b \rangle = (-1)^{n-1} \langle a, b \rangle.$$
 By taking $a$ in the $\varepsilon$-isotypic component and $b$ in $\eta$-isotypic component we see that 
 $$
 \varepsilon(\delta_n)\eta(\delta_n) \langle a, b\rangle = \langle \delta_n \cdot a, \delta_n \cdot b \rangle = (-1)^{n-1} \langle a, b \rangle.
 $$ 
 Hence, our claim at the beginning of the proof is true.
\end{proof}

The above proposition motivates the notation: 
$$
\tilde \varepsilon = (-1)^{n-1}\varepsilon.
$$
Observe that $\varepsilon$ is permissible for $\pi$ if and only if $\tilde \varepsilon$ is permissible for 
$\tilde \pi.$

\smallskip
\subsubsection{\bf A nonvanishing hypothesis at infinity}\label{nv-hyp}

The following proposition will be proved in \S \ref{nv-sec} using the methods initiated by Sun~\cite{Sun}. 

\begin{proposition}
\label{prop:nv-hyp}
$\fr c_\infty (\pi, \tilde \pi, \varepsilon, \tilde \varepsilon) \not = 0$. 
\end{proposition}

We now define a {\it period at infinity} as:  
\begin{equation}
\label{eqn:period-infty}
\fr p_\infty (\pi) \ := \  \left(\fr c_\infty (\pi, \tilde \pi, \varepsilon, \tilde \varepsilon)\right)^{-1}.
\end{equation}
We remark here that there is only one permissible character $\varepsilon$ when $n$ is odd, and that this quantity does not depend on $\varepsilon$ when $n$ is even (see the last paragraph of \S\ref{nv-sec}), 
justifying our notation $\fr p_\infty (\pi)$ not involving $\varepsilon.$ 
We mention the following result {\it en passant}. 

\begin{proposition}
Suppose $F$ is totally real or a totally imaginary quadratic extension of a totally real field, then
$$
\fr p_\infty ({}^\sigma\!\pi) = \fr p_\infty (\pi).
$$
\end{proposition}

\begin{proof}
This follows from the K\"unneth theorem for relative Lie-algebra cohomology applied after 
the fact that when the base field $F$ is as in the statement, then the action of $\sigma$ on 
$\pi_\infty$ is basically via permutations; see Gan--Raghuram \cite[Prop.\,3.2]{gan-raghuram}. 
\end{proof}

\smallskip
\subsubsection{\bf The main result on adjoint $L$-values}
\label{proof-thma}

\begin{theorem}
\label{thm:l-value}
Let $\pi \in {\rm Coh}(G, \lambda, K_f)$ and let $\varepsilon$ be a permissible signature for $\pi.$ Define:
\begin{equation}\label{L-alg}
\Lalg (1, \Ad^0, \pi, \varepsilon) \ := \ 
\frac{L(1, \Ad^0, \pi)}{\omega_F \cdot \fr p_\ram (\pi) \cdot \fr p_\infty (\pi) \cdot \fr 
p^\varepsilon (\pi) \cdot \fr q^{\tilde \varepsilon} (\tilde \pi)} .
\end{equation}
For all $\sigma \in \Aut(\C)$ we have
$$
\sigma( \Lalg (1, \Ad^0, \pi, \varepsilon)) \ = \  \Lalg (1, \Ad^0, {^\sigma} \pi, \varepsilon).
$$

In particular, $\Lalg (1, \Ad^0, \pi, \varepsilon)  \in \Q(\pi).$ Furthermore, keeping in mind the refined periods, we
have $\Lalg (1, \Ad^0, \pi, \varepsilon)  \in \calo.$
\end{theorem}

\begin{proof}
We know from (\ref{eqn:refined-periods}), (\ref{eqn:l-value-pairing}), (\ref{eqn:l-value-integral-pairing}), and 
(\ref{eqn:period-infty}) that 
\begin{equation}
\label{eqn:L-alg-value-pairing}
\Lalg (1, \Ad^0, \pi, \varepsilon) \ = \ 
\langle \vartheta^\circ_{b,\varepsilon}, \tilde \vartheta_{t, \tilde \varepsilon}^\circ \rangle \in \calo. 
\end{equation}
By definition of the periods, we also know that 
$$
{^\sigma} \vartheta_{b, \varepsilon}^\circ (\pi) = \vartheta_{b, \varepsilon}^\circ ({^\sigma} \pi) 
\quad \mathrm{and} \quad 
{^\sigma} \tilde \vartheta_{t,  \eta}^\circ (\tilde \pi) = \tilde \vartheta_{t, \eta}^\circ ({^\sigma} \tilde \pi).
$$
Applying $\sigma$ to (\ref{L-alg}) above gives: 
\begin{equation*}
\begin{split}
\sigma (\Lalg (1, \Ad^0, \pi, \varepsilon))  
& \ = \ \sigma(\langle \vartheta^\circ_{b,\varepsilon} (\pi) , \tilde \vartheta_{t, \tilde \varepsilon}^\circ (\tilde{\pi}) \rangle) \\
& \ = \ \langle {^\sigma} \vartheta^\circ_{b,\varepsilon} (\pi) , {^\sigma}\tilde \vartheta_{t, \tilde \varepsilon}^\circ (\tilde{\pi}) \rangle \\
& \ = \ \langle \vartheta_{b, \varepsilon}^\circ ({^\sigma} \pi) , 
\tilde \vartheta_{t,  \tilde \varepsilon}^\circ ({^\sigma} \tilde{\pi}) \rangle \\
& \ = \ \Lalg (1, \Ad^0, {^\sigma}\!\pi, \varepsilon).
\end{split}
\end{equation*}
Here the second equality follows from the $\sigma$-equivariance of Poincar\'e duality.
\end{proof}

The following piquant period relation, which is an easy corollary of the above theorem, 
is a generalization of a result of Shimura for Hilbert modular forms 
\cite[Thm.\,4.3, II]{shimura-duke} that says that the product 
$c^\varepsilon({\bf f}) c^{-\varepsilon}({\bf f})$ of a period and its `complementary' period attached to a holomorphic cuspidal Hilbert modular form ${\bf f}$ is essentially the Petersson norm of ${\bf f}$. Such a period relation was hinted at in the last paragraph of the article \cite{raghuram-legacy}.

\begin{cor}
Suppose the base field $F$ has at least one real place, and suppose $n$ is even.  
Let $\pi \in {\rm Coh}(G, \lambda, K_f).$ For any two characters $\varepsilon_1$ and $\varepsilon_2$ we have:
$$
\fr p^{\varepsilon_1} (\pi) \cdot \fr q^{-\varepsilon_1} (\tilde \pi) 
\ \approx \ 
\fr p^{\varepsilon_2} (\pi) \cdot \fr q^{-\varepsilon_2} (\tilde \pi), 
$$
where by $\approx$ we mean up to an element of the number field $\Q(\pi).$ Moreover, the ratio of the left hand side by the right hand side is $\Aut(\C)$-equivariant. 

\end{cor}

\medskip

\subsection{On the criticality of the adjoint $L$-function at $s=1$}
Following Deligne \cite{deligne}, we will say $s=s_0$ is critical if and only if $L_\infty (s, \Ad^0, \pi)$ and 
$L_\infty (1-s, (\Ad^0)^\sv, \pi)$ are regular (i.e., no poles) at $s=s_0$. Since $(\Ad^0)^\sv = \Ad^0$, we see that $s=1$ is critical if and only if both $L_\infty (1, \Ad^0, \pi)$ and $L_\infty (0, \Ad^0, \pi)$ are regular.

\begin{proposition}
\label{prop:critical}
The point $s=1$ is critical for $L(s, \Ad^0, \pi)$ if and only if $n=2$ and $F$ is totally real. 
\end{proposition}
\begin{proof}
Suppose $F$ has a complex place, say  $v$. Let $\sr_L(\pi_v)$ denote the Langlands parameter--which is an $n$-dimensional semi-simple representation of the Weil group $W_\C \simeq \C^\times$ of $\C$--of the representation $\pi_v$ of $\GL_n(\C).$ Using the local Langlands correspondence for $\GL_n(\C)$ (see \cite{knapp}) and 
explicit knowledge of $\pi_v$ (see, for example, \cite[Sect.\,2.4.2]{raghuram-2013}), we have $\sr_L(\pi_v) = \oplus_{i=1}^n z^{a_i} \bar z^{b_i}$ for half-integers 
$a_i, b_i.$ Hence 
$$
\sr_L(\pi_v) \otimes \sr_L(\tilde \pi_v) \ = \ \bigoplus_{1 \le i,j \le n} z^{(a_i -a_j)} \bar z^{(b_i-b_j)}.
$$ 
We know that $L(s, z^p \bar z^q) = 2 (2\pi)^{-(s+p_0)} \Gamma (s+p_0)$ where $p_0=\text{max}\{p,q\}$. The trivial character $\id$ appears in $\sr_L(\pi) \otimes \sr_L(\tilde \pi)$ at least $n$ times. This implies that $L_\infty (s, \Ad^0, \pi)$ has a factor of the form $\Gamma (s)^{n-1}$.  Since $n \ge 2$, this implies that $L_\infty (s, \Ad^0, \pi)$ is not regular at $s=0$. Hence, if $F$ has a complex place then $s=1$ is not critical for $L(s, \Ad^0, \pi).$

Now suppose that $F$ is totally real. Then using \cite{knapp} and 
\cite[\S\,2.4.1]{raghuram-2013}, the Langlands parameters of $\pi_v$ and $\tilde \pi_v$ for any infinite place $v$ 
are given by: 
\begin{equation*}
\begin{split}
\sr_L(\pi_v) &= I(\chi_{l_1})(-w/2) \oplus \dots \oplus  I(\chi_{l_{n/2}})(-w/2), \text{ and } \\
\sr_L(\tilde \pi_v) &= I(\chi_{l_1})(w/2) \oplus \dots \oplus  I(\chi_{l_{n/2}})(w/2)
\end{split}
\end{equation*}
when $n$ is even, and 
\begin{equation*}
\begin{split}
\sr_L(\pi_v) &= I(\chi_{l_1})(-w/2) \oplus \dots \oplus  I(\chi_{l_{[n/2]}})(-w/2) \oplus \varepsilon (-w/2), \text{ and } \\
\sr_L(\tilde \pi_v) &= I(\chi_{l_1})(w/2) \oplus \dots \oplus  I(\chi_{l_{[n/2]}})(w/2) \oplus \varepsilon (w/2)
\end{split}
\end{equation*}
when $n$ is odd; where, $l_j \geq 1$ and $w$ are integers, and $I(\chi_l)$ stands for the representation of the Weil group $W_\R$ of $\R$ induced from the character $\chi_l$ of $\C^\times$ that sends $z = re^{i\theta}$ to 
$e^{il\theta}.$ Consider two sub-cases: 
\begin{itemize}
\item
Suppose $F$ is totally real and $n\ge 3$, then $\id$ appears at least twice in $\sr_L (\pi_v) \otimes \sr_L (\tilde \pi_v)$. Hence $\Gamma (s/2)$ is a factor of $L_\infty (s, \Ad^0, \pi)$ and $s=1$ is not critical for $L(s, \Ad^0, \pi).$ 

\item
Finally, suppose $F$ is totally real and $n=2$, then we see that 
$$\sr_L(\pi_v) \otimes \sr_L(\tilde \pi_v) = I(\chi_{2l_1}) \oplus (\text{sgn}) \oplus \id.$$ 
This implies that up to nonzero constants and exponential factors, 
$$
L_\infty (s, \Ad^0, \pi) \sim \Gamma (s+l_1) \Gamma \left(\frac{s+1}{2} \right). 
$$
Hence $s=1$ is critical. 
\end{itemize}
This completes the proof of the proposition.
\end{proof}

\medskip

\subsubsection{\bf A noncritical value for symmetric fourth $L$-function of $\GL_2$}
Let's record an interesting consequence of Thm.\,\ref{thm:l-value} for certain symmetric power $L$-functions. 
The reader should consult \cite{raghuram-2013} for details concerning symmetric power transfers and symmetric power $L$-functions. 

\smallskip
 
Let $\pi$ be a cohomological cuspidal automorphic representation of $\GL_2/F,$ and assume that 
the central character of $\pi$ is trivial. In particular $\pi = \tilde \pi$ and $L(1, {\rm Ad}^\circ, \pi) = L(s, \Sym^2(\pi)).$ 
 We have: 
\begin{equation}
\label{eqn:sym-2}
L(1, \Sym^2(\pi)) \ \approx \ 
\omega_F \cdot \fr p_\ram (\pi) \cdot \fr p_\infty (\pi) \cdot \fr p^\varepsilon (\pi) \cdot \fr q^{-\varepsilon} (\pi), 
\end{equation}
for any character $\varepsilon$, and by $\approx$ we mean up to an element of the rationality field $\Q(\pi).$ 
Note that $L(1, \Sym^2(\pi))$ is critical if and only if $F$ is totally real. Furthermore, $L(1, \Sym^2(\pi)) \neq 0.$ 

\smallskip

Assume $\pi$ is not dihedral, and so its symmetric square transfer $\Sym^2(\pi)$ is also a cuspidal automorphic representation; see \cite{gelbart-jacquet}. Also, as explained in \cite{raghuram-2013}, we know that 
$\Sym^2(\pi)$ is cohomological. Since $\pi$ is self-dual, so is $\Sym^2(\pi)$; we have:
$$
L(s,  \Sym^2(\pi) \times \Sym^2(\pi)) \ = \
L(s,  \Sym^4(\pi)) L(s,  \Sym^2(\pi)) \tilde\zeta_F(s), 
$$
which may be checked by checking equality locally everywhere. In particular, we have: 
$$
L(s, \Ad^\circ, \Sym^2(\pi)) \ = \ L(s,  \Sym^4(\pi)) L(s,  \Sym^2(\pi)), 
$$
and they are all nonzero at $s=1.$ Applying Thm.\,\ref{thm:l-value} to $L(s, \Ad^\circ, \Sym^2(\pi))$ and using (\ref{eqn:sym-2}) we get
\begin{equation}
\label{eqn:sym-4}
L(1,  \Sym^4(\pi)) \ \approx \  
\frac{\fr p_\ram (\Sym^2(\pi)) \cdot \fr p_\infty (\Sym^2(\pi)) \cdot 
\fr p^\eta (\Sym^2(\pi)) \cdot \fr q^{\eta} (\Sym^2(\pi))}
{\fr p_\ram (\pi) \cdot \fr p_\infty (\pi) \cdot \fr p^\varepsilon (\pi) \cdot \fr q^{-\varepsilon} (\tilde \pi)}, 
\end{equation}
where we take $\eta$ to be the permissible character for $\Sym^2(\pi).$ Note that both 
$L(1, \Ad^\circ, \Sym^2(\pi))$ and $L(1,  \Sym^4(\pi))$ are non-critical $L$-values. 

\smallskip

Clearly, both (\ref{eqn:sym-2}) and (\ref{eqn:sym-4}) can be stated in a Galois-equivariant fashion.  
Using results in \cite{raghuram-imrn} and \cite{raghuram-2013}, we may proceed in this fashion, and by induction, 
get algebraicity results for 
$$
L(1, \Sym^{2m}(\pi)), 
$$
which are unconditional for $m \leq 4$, and depend on unproven instances of Langlands's functoriality for higher even symmetric powers.

\bigskip

\section{\bf Discriminant calculations and cohomological congruences}

\subsection{Some linear algebra}
\label{lattice-cong-mod}

Recall that $E$ is an extension of $\Q_p$ as in the introduction, and denote the residue field $\calo/\wp$ by $\kappa.$
Let $V$ and $\tilde V$ be finite-dimensional vector spaces over a field $E$. Let $L$ and $\tilde L$ be $\calo$-lattices in $V$ and $\tilde V,$ respectively. Consider a non-degenerate bilinear form $\langle \ , \ \rangle : V \times \tilde V \to E$. For the lattice $L$, define its dual lattice to be $L^* = \{ v \in \tilde V \mid \langle v, w \rangle \in \calo,\  \forall w \in L \}$. One can similarly define $\tilde L^*$, the dual lattice for $\tilde L$ in $V$.  We say that the pairing restricts to a perfect pairing $L \times \tilde L \to \calo$ if $L^* = \tilde L$ and $\tilde L^* = L$.

Let $V=V_1 \oplus V_2$ and $\tilde V = \tilde V_1 \oplus \tilde V_2$ be decompositions of $V$ and $\tilde V$ into subspaces that respect the pairing above (i.e., $V_1 \perp \tilde V_2$ and $V_2 \perp \tilde V_1$). Let $\pi_i$ be the projections of $V$ onto $V_i$ and let $\tilde \pi_i$ be projections of $\tilde V$ onto $\tilde V_i$.  For $i=1,2$, define
$$L_i = L \cap V_i, \ \tilde L_i = \tilde L \cap \tilde V_i \quad \text{and}\quad \Lambda_i = \pi_i (L), \ \tilde \Lambda_i = \tilde \pi_i (\tilde L).$$
We are interested in the congruence module for the lattice $L$ with respect to this decomposition. This is defined as 
$$
\calc (L; V_1, V_2) \ := \ \frac{\Lambda_1 \oplus \Lambda_2}{L}  \ \cong \  \frac{L}{L_1 \oplus L_2}.
$$
We know that the projection maps induce the following isomorphisms:
$$
\xymatrix@1{ \frac{\Lambda_1}{L_1} \ & \ \frac{L}{L_1 \oplus L_2} \ar[l]_-{\pi_1} \ar[r]^-{\pi_2} \ & \ 
\frac{\Lambda_2}{L_2}. }
$$
In order to produce congruences we need to show that $\calc(L; V_1, V_2) \not = 0$. We know that
$$
\text{disc}(L_1 \times \tilde L_1) \ =_{\calo^\times}  \ |\tilde L_1^*/L_1|, 
$$
where $\tilde L_1^*= \{ v \in V_1 \mid \langle v, w \rangle \in \calo,\  \forall w \in \tilde L_1 \}$. The following lemma shows that the congruence module is nonzero if and only if $v_\wp ( \text{disc} (L_1 \times \tilde L_1)) >0$.

\begin{lemma}
$\Lambda_1 = \tilde L_1^*$.
\end{lemma}
\begin{proof}
One can check that $\Lambda_1 \subset \tilde L_1^*$. Let $u_1, \dots, u_t$ be a basis for $\tilde L$ such that $u_1, \dots, u_s$ is a basis for $\tilde L_1$. Let $u_1^*, \dots, u_t^*$ be the corresponding dual basis for the dual lattice $L$. It is then easy to check that $\pi_1 (u_1^*), \dots ,\pi_1 (u_s^*)$ is a basis for $V_1$ and they generate the lattice $\Lambda_1$ over $\calo$. Let $v \in \tilde L_1^*$, write $v = \sum a_i \pi_1 (u_i^*)$ with $a_i \in E$. Then $\calo \ni \langle v, u_i \rangle = \sum_j a_j \langle \pi_1(u^*_j), u_i \rangle = a_i$.
\end{proof}

Now let $\calh^\circ \subset \End_\calo (L)$ be the subalgebra generated by a collection of pairwise commuting operators acting on $L$ and let $\calh^\circ_E = \calh^\circ \otimes E \subset \End_E(V)$. The decomposition $V=V_1\oplus V_2$ gives rise to idempotents $e_1, e_2 \in \calh_E^\circ$. Let $\calh_i^\circ = \calh^\circ \cap e_i \calh^\circ_E$, and we define the congruence module
$$
Q(\calh^\circ; e_1, e_2) \ = \ \frac{\calh^\circ}{\calh^\circ_1 \oplus \calh^\circ_2} \ \cong \ \frac{e_i \calh^\circ}{\calh^\circ_i}.
$$ 

There is a natural map $Q(\calh^\circ; e_1, e_2) \to \End_\calo (\calc(L; V_1, V_2)),$ and the image of this map is the subalgebra generated by the operators induced by $T \in \calh^\circ$. Clearly, we have
\begin{equation}
\label{eqn:strong-congruence}
\calc(L; V_1, V_2)\not = 0  \ \implies \ Q(\calh^\circ; e_1, e_2)\not = 0.
\end{equation}
 Suppose $Q(\calh^\circ; e_1, e_2) \not = 0$, then there there exist maximal ideals $\fr M, \fr M_1\ \mathrm{ and }\ \fr M_2$ such that 
$$
\calh^\circ_1 \oplus \calh^\circ_2  \ \subset \ \fr M \subset \calh^\circ,  \quad
\calh^\circ_1 \ \subset \ \fr M_1 \subset e_1\calh^\circ,   
\quad \mathrm{ and } \quad 
\calh^\circ_2 \ \subset \  \fr M_2 \ \subset \ e_2\calh^\circ.
$$ 
We then have the following commutative diagram:
$$\xymatrix@1{
 & & \calh^\circ \ar[rrd] \ar[lld] & & \\
 e_1 \calh^\circ \ar[ddd] \ar[rrd]^{\chi_1} \ar[rrdd]_{\bar \chi_1} & & & & e_2 \calh^\circ \ar@{-->}[lld]_{\chi_2} \ar[ddd] \ar[lldd]^{\bar \chi_2}  \\
 & & \calo \ar[d] & &  \\
 & & \kappa & &  \\
 e_1\calh^\circ/\fr M_1 \ar@{=}[rrrr]\ar[rru]^{\sim}  & & & &  e_2\calh^\circ/\fr M_2 \ar[llu]_{\sim}  \\
}$$

The maps $\chi_1$ and $\bar \chi_1$ are induced by $\pi$ and we define $\bar \chi_2$ by the commutativity of the diagram. A lemma of Deligne and Serre (\cite[Lemma 6.11]{DS}) gives us a character $\chi_2$ that lifts $\bar \chi_2$ and a Hecke eigenclass in $\Theta \in L_2$ such that the Hecke operators act on $\Theta$ via $\chi_2$, after possibly enlarging $E.$

\medskip

\subsection{Discriminants and congruence modules}

Let $\lambda$ be a strongly-pure dominant integral weight, and $\pi \in {\rm Coh}(G, K_f, \lambda)$ as before. 

\subsubsection{\bf Case I}
We consider the following finite-dimensional vector spaces over $E$:
$$ 
V \ = \ H^b_! (\SGK, \tcalm_{\lambda, E}) 
\quad \text{and} \quad 
\tilde V \ = \ H^t_! (\SGK, \tcalm_{ \tilde \lambda, E}) .
$$ 
Let $V_1$ and $\tilde V_1$ be the following subspaces of $V$ and $\tilde V,$ respectively:
$$ 
V_1 \ = \ \bigoplus_\varepsilon H^b_! (\SGK, \tcalm_{\lambda, E}) (\pi_f \times  \varepsilon) 
\quad \text{and} \quad 
\tilde V_1 \ = \ \bigoplus_{ \varepsilon} H^t_! (\SGK, \tcalm_{ \tilde \lambda, E}) (\tilde \pi_f \times \tilde \varepsilon), 
$$
where the direct sum is over characters $\varepsilon$ on $K_\infty/ K_\infty^\circ$ that are permissible for $\pi$. 
Let $V_2$ (resp., $\tilde V_2$) be a Hecke complement of $V_1$ (resp., $\tilde V_1$) in $V$ (resp., $\tilde V$). Let $L$ and $\tilde L$ be the following lattices in $V$ and $\tilde V,$ respectively:
$$ 
L \ = \ \Hbar^b_! (\SGK, \tcalm_{\lambda, \calo}) 
\quad \text{and} \quad 
\tilde L \ =\  \Hbar^t_! (\SGK, \tcalm_{\tilde \lambda, \calo}).
$$ 
We see that 
\begin{equation*}
\begin{split}
L_1 & \ = \ L \cap V_1 \ = \ \bigoplus_\varepsilon \Hbar^b (\SGK, \tcalm_{\lambda, \calo}) (\pi_f \times \varepsilon),\quad \mathrm{and}\\
\tilde L_1 & \ = \ \tilde L \cap \tilde V_1 \ = \ 
\bigoplus_{\varepsilon} \Hbar^t (\SGK, \tcalm_{\tilde{\lambda}, \calo}) (\tilde \pi_f \times \tilde \varepsilon).
\end{split}
\end{equation*}

Cup product induces the following bilinear forms:
$$\langle \ , \ \rangle : V \times \tilde V \to E \quad \text{and} \quad \langle \ , \ \rangle : L \times \tilde L \to \calo.$$
The first pairing is non-degenerate. We assume that the second pairing is perfect, i.e., the induced maps $L \to \Hom_\calo (\tilde L, \calo)$ and $\tilde{L} \to \Hom_\calo (L, \calo)$ are isomorphisms. There is a finite set 
of rational primes $S_1$ outside which the perfectness assumption holds; 
see the discussion in \S \ref{excluded-primes} below, and especially see (\ref{eqn:S1}). 

This puts us in the situation of \S \ref{lattice-cong-mod}, and we consider the congruence module $\calc (L; V_1, V_2)$ of the lattice associated to this decomposition. We have chosen an $\calo$-basis 
$\{ \vartheta_{b, \varepsilon}^\circ (\pi) \}_\varepsilon $ (resp., 
$\{ \tilde \vartheta_{t, \tilde \varepsilon}^\circ (\tilde \pi) \}_{\varepsilon}$) for $L_1$ (resp., $\tilde L_1$). The discriminant with respect to this basis is 
\begin{equation}
\label{eqn:disc-and-theta-pairing}
\mathrm{disc}(L_1 \otimes \tilde L_1) \ = \ \prod_\varepsilon \langle  \vartheta_{b,\varepsilon}^\circ (\pi), \tilde \vartheta_{t, \tilde \varepsilon}^\circ (\tilde \pi) \rangle, 
\end{equation}
since the $( \pi \times \varepsilon)$-component of $V$ pairs non-trivially only with the $(\tilde \pi \times \tilde \varepsilon)$-component of $\tilde V$. This shows that the congruence module for $L$ is zero if and only if the above 
discriminant is a unit in $\calo.$ If for some $\varepsilon$,
\begin{equation}
\begin{split}
v_\wp (\Lalg (1, \Ad^0, \pi, \varepsilon)) >0  
& \ \implies \ v_\wp (\prod_\varepsilon \Lalg (1, \Ad^0, \pi, \varepsilon)) >0 \\
& \ \iff \  v_\wp (\text{disc} (L_1 \otimes \tilde L_1)) >0, \quad  \mbox{(by 
(\ref{eqn:L-alg-value-pairing}) and (\ref{eqn:disc-and-theta-pairing}))} \\ 
& \ \iff \  \wp \in \text{Supp}_\calo (\calc (L; V_1, V_2)).
\end{split}								   
\end{equation}
Now using the discussion leading up to (\ref{eqn:strong-congruence}) we get 
$\wp \in \text{Supp}_\calo (Q(\calh^\circ; e_1, e_2)).$ This implies that there is Hecke module $\pi'_f$ contributing to 
$V_2$ which is congruent to $\pi_f$; this $\pi'_f$ corresponds to the character $\chi_2$ produced by the Deligne--Serre lemma as above. In other words, we get that an automorphic representation $\pi'$ whose finite part contributes to inner cohomology, and by definition of $V_1$ we know that $\pi' \not\simeq \pi$, and such that $\pi' \equiv \pi \pmod{\wp}.$

\medskip
\subsubsection{\bf Case II}
We now assume that the weight $\lambda$ is parallel, see \cite[\S 2.3.4]{raghuram-2013} for the precise definition. This implies that for any $\sigma \in \Aut(\C),$ the representation 
${^\sigma} \pi$ is also cohomological with respect to the weight $\lambda$.  Consider the finite-dimensional vector spaces over $E$:
$$ 
V \ = \ H^b_! (\SGK, \tcalm_{\lambda, E}) \quad \text{and} \quad 
\tilde V \ = \ H^t_! (\SGK, \tcalm_{ \tilde \lambda, E}). 
$$ 
Let $V_1$ and $\tilde V_1$ be the following subspaces of $V$ and $\tilde V,$ respectively:
$$ 
V_1 \ = \ \bigoplus_{\sigma, \varepsilon} H^b_! (\SGK, \tcalm_{\lambda, E}) ({^\sigma} \pi_f \times \varepsilon) 
\quad \text{and} \quad 
\tilde V_1 \ = \ \bigoplus_{\sigma, \varepsilon} H^t_! (\SGK, \tcalm_{ \tilde \lambda, E}) 
({^\sigma} \tilde \pi_f \times  \tilde \varepsilon), 
$$
where the direct sum is over characters $\varepsilon$ on $K_\infty/ K_\infty^\circ$ that are permissible for $\pi$ and over
$\sigma$ which runs over the finite set of all embeddings $\Q (\pi) \to \C$. 
Let $V_2$ (resp., $\tilde V_2$) be a Hecke-complement of $V_1$ (resp., $\tilde V_1$) in $V$ (resp., $\tilde V$). 
Let $L$ and $\tilde L$ be the following lattices in $V$ and $\tilde V,$ respectively:
$$ 
L \ = \ \Hbar^b_! (\SGK, \tcalm_{\lambda, \calo}) 
\quad \text{and} \quad 
\tilde L \ = \ \Hbar^t_! (\SGK, \tcalm_{\tilde \lambda, \calo}).
$$ 
We see that 
\begin{equation*}
\begin{split}
L_1 & \ = \ L \cap V_1 \supseteq L(\pi) \ := \ 
\bigoplus_{\sigma, \varepsilon} \Hbar^b_! (\SGK, \tcalm_{\lambda, \calo}) ({^\sigma} \pi_f \times \varepsilon),\quad \mathrm{and}\\
\tilde L_1 & \ = \ \tilde L \cap \tilde V_1 \supseteq  L(\tilde \pi) \ := \ 
\bigoplus_{\sigma, \varepsilon} \Hbar^t_! (\SGK, \tcalm_{ \tilde \lambda, \calo}) 
({^\sigma} \tilde \pi_f \times  \tilde \varepsilon).
\end{split}
\end{equation*}
We further assume that after excluding a finite set of rational primes if necessary the above inclusions 
are equalities. 
As before, cup product induces the following bilinear forms:
$$\langle \ , \ \rangle : V \times \tilde V \to E \quad \text{and} \quad \langle \ , \ \rangle : L \times \tilde L \to \calo.$$
The first pairing is non-degenerate and the second pairing is perfect after excluding a finite set $S_2$ of primes, where\begin{equation}
\label{eqn:S2}
S_2 = S_1 \cup \{p \ : \  L(\pi) \varsubsetneq L_1 \}.
\end{equation}

We now have $\calo$-bases $\{ \vartheta_{b, {^\sigma} \varepsilon}^\circ ({^\sigma}\pi) \}_{\sigma, \varepsilon} $  
and $\{ \tilde \vartheta_{t, \tilde \varepsilon}^\circ ({^\sigma} \tilde \pi) \}_{\sigma, \varepsilon}$ for $L(\pi)$ 
and $L(\tilde \pi)$ respectively. Calculating the discriminant with respect to this basis, we see that
$$
\mathrm{disc}(L(\pi) \otimes L(\tilde \pi)) \ = \ \prod_{\sigma, \varepsilon} 
\langle  \vartheta_{b, \varepsilon}^\circ ({^\sigma} \pi), 
\tilde \vartheta_{t, \tilde \varepsilon}^\circ ({^\sigma} \tilde \pi) \rangle, 
$$
since the $( {^\sigma} \pi \times \varepsilon)$-component of $V$ pairs non-trivially with the 
$({^\tau} \tilde \pi \times \tilde \varepsilon)$-component of $\tilde V$ if and only if $\sigma = \tau$. Since the discriminants of $L(\pi) \otimes L(\tilde \pi)$ and $L_1 \otimes \tilde L_1$ differ by a unit in $\calo$, for any $\varepsilon$ we have
\begin{equation*}
\begin{split}
v_\wp (\Lalg (1, \Ad^0, \pi, \varepsilon)) >0  & \implies v_\wp (\prod_{\varepsilon} 
\Lalg (1, \Ad^0, {^\sigma} \pi,  \varepsilon)) >0 \\
								   & \iff  v_\wp (\text{disc} (L_1 \otimes \tilde L_1)) > 0 \\
								   & \iff  \wp \in \text{Supp}_\calo (\calc (L; V_1, V_2)).
\end{split}								   
\end{equation*}
As before, we get that an automorphic representation $\pi'$ whose finite part contributes to $V_2$, and by definition of $V_1$ we know that $\pi' \not\simeq {}^\sigma\!\pi$, and such that $\pi' \equiv \pi \pmod{\wp}.$

\medskip
\subsubsection{\bf Case III}
We continue with our assumption that $\lambda$ is parallel and consider the following finite-dimensional vector spaces over $E$.
$$ 
V \ = \ \bigoplus_{\sigma, \varepsilon} H^b_! (\SGK, \tcalm_{\lambda, E}) ({^\sigma} \pi_f \times \varepsilon) \quad \text{and} \quad 
\tilde V \ = \ \bigoplus_{\sigma, \varepsilon} H^t_! (\SGK, \tcalm_{ \tilde \lambda, E}) 
({^\sigma} \tilde \pi_f \times  \tilde \varepsilon).
$$
Let $V_1$ and $\tilde V_1$ be the following subspaces of $V$ and $\tilde V,$ respectively:
$$ 
V_1 \ = \ \bigoplus_\varepsilon H^b_! (\SGK, \tcalm_{\lambda, E}) (\pi_f \times  \varepsilon)
 \quad \text{and} \quad 
\tilde V_1 \ = \ \bigoplus_{\varepsilon} H^t_! (\SGK, \tcalm_{ \tilde \lambda, E}) (\tilde \pi_f \times \tilde \varepsilon).
$$
(This case is interesting only when $\Q \subsetneq \Q(\pi)$ which makes $V_1 \subsetneq V$; 
see Rem.\,\ref{rem:explanations} below.)
Let $V_2$ (resp., $\tilde V_2$) be a Hecke-complement of $V_1$ (resp., $\tilde V_1$) in $V$ (resp.\,$\tilde V$).
Let $L$ and $\tilde L$ be the following $\calo$-lattices in $V$ and $\tilde V,$ respectively:
$$ 
L \ = \ \Hbar^b_! (\SGK, \tcalm_{\lambda, \calo}) \cap V 
\quad \text{and} \quad 
\tilde L \ = \ \Hbar^t_! (\SGK, \tcalm_{\tilde \lambda, \calo}) \cap \tilde V.
$$ 
We see that 
$$
L_1  \ = \ L \cap V_1 \ = \bigoplus_\varepsilon \Hbar^b (\SGK, \tcalm_{\lambda, \calo}) (\pi_f \times \varepsilon),
\quad \mathrm{and} \quad 
\tilde L_1  \ = \ \tilde L \cap \tilde V_1 \ = \bigoplus_{\varepsilon} \Hbar^t (\SGK, \tcalm_{\tilde{\lambda}, \calo})
(\tilde \pi_f \times \tilde \varepsilon).
$$

Cup product induces the following bilinear forms:
$$
\langle \ , \ \rangle : V \times \tilde V \to E \quad \text{and} \quad \langle \ , \ \rangle : L \times \tilde L \to \calo.
$$
The first pairing is non-degenerate and we assume that the second pairing is perfect after excluding a finite set 
$S_3$ of primes. Note that by definition $S_1 \subset S_3$. Calculating the discriminant with respect to the bases $\{ \vartheta_{b, \varepsilon}^\circ (\pi) \}_{\varepsilon} $  and $\{ \tilde \vartheta_{t, \tilde \varepsilon}^\circ (\tilde \pi) \}_{\varepsilon}$ for $L_1$ and $\tilde L_1$,  we see that
$$\mathrm{disc}(L_1 \otimes \tilde L_1) = \prod_\varepsilon \langle  \vartheta_{b,\varepsilon}^\circ (\pi), \tilde \vartheta_{t, \tilde \varepsilon}^\circ (\tilde \pi) \rangle$$
since the $( \pi \times \varepsilon)$-component of $V$ pairs non-trivially only with the $(\tilde \pi \times \tilde \varepsilon)$-component of $\tilde V$. This shows, for any $\varepsilon$, that
\begin{equation*}
\begin{split}
v_\wp (\Lalg (1, \Ad^0, \pi, \varepsilon)) >0  & \ \implies \ v_\wp (\prod_\varepsilon \Lalg (1, \Ad^0, \pi, \varepsilon)) >0 \\
								   & \ \iff \ v_\wp (\text{disc} (L_1 \otimes \tilde L_1)) > 0 \\
								   & \ \iff \ \wp \in \text{Supp}_\calo (\calc (L; V_1, V_2)).
\end{split}								   
\end{equation*}
These are precisely the primes in $S_2 \setminus S_1$ and hence Case II and Case III do not occur simultaneously, i.e., 
\begin{equation}
\label{eqn:S3}
S_3 \ = \ S_1 \cup \{\mbox{the congruence primes from {\rm Case II}} \}.
\end{equation}
As before, we get that an automorphic representation $\pi'$ whose finite part contributes to $V_2$, and by definition of $V_1$, this means that we get an embedding $\sigma : \Q(\pi) \to \C$ such that $\pi' = {}^\sigma\!\pi \not\simeq \pi$, and such that $\pi' \equiv \pi \pmod{\wp}.$

\medskip
\subsubsection{\bf The excluded primes}\label{excluded-primes}

We briefly describe the set $S_1$ of excluded primes. This calculation is similar to \cite[Thm.\,3]{Gh}. Recall that the set $S_1$ is the set of primes $\wp$ such that the Poincar\'e pairing between the following integral cohomology groups
$$ 
\Hbar^b_! (\SGK, \tcalm_{\lambda, \calo}) \quad \text{and} \quad 
\Hbar^t_! (\SGK, \tcalm_{\tilde \lambda, \calo})
$$
is perfect. The paring induces the following maps:
\begin{equation}\label{perfectness}
\begin{split}
\Hbar^b_! (\SGK, \tcalm_{\lambda, \calo}) &\to \Hom_\calo ( \Hbar^t_! (\SGK, \tcalm_{\tilde \lambda, \calo}), \calo), \text{ and } \\
\Hbar^t_! (\SGK, \tcalm_{\tilde \lambda, \calo}) &\to \Hom_\calo ( \Hbar^b_! (\SGK, \tcalm_{\lambda, \calo}), \calo).
\end{split}
\end{equation}
Since these are free $\calo$-modules of finite rank, it suffices to show that one of these maps is an isomorphism and we pick the first one. This map is injective because the corresponding map at the level of rational cohomology
$$
\Hbar^b_! (\SGK, \tcalm_{\lambda, E}) \to \Hom_E ( \Hbar^t_! (\SGK, \tcalm_{\tilde \lambda, E}), E)
$$
is injective. Moreover, by \cite[Thm.\,4.8.9]{Ha-AlGeo1}, we have the following isomorphism
$$
\Hbar^b (\SGK, \tcalm_{\lambda, \calo}) \to \Hom_\calo ( H^t_c (\SGK, \tcalm_{\tilde \lambda, \calo}), \calo). 
$$
Using the long exact sequence coming from the Borel--Serre compactification of $\SGK$ 
(see, for example, \cite[\S 2]{harder-raghuram-preprint}), we see that the obstruction to surjectivity in (\ref{perfectness}) is a torsion element of
$$
\frac{H^b (\SGK, \tcalm_{\lambda, \calo}) }{H^b_! (\SGK, \tcalm_{\lambda, \calo}) } 
\subset H^b (\partial \SGK, \tcalm_{\lambda, \calo}), 
$$
where $\partial \SGK$ is the Borel--Serre boundary of the compactification. Hence this pairing is perfect if $p \not \in S_1$, where 
\begin{equation}
\label{eqn:S1}
S_1 = \{ p \mid H^b (\partial \SGK, \tcalm_{\lambda, \calo}) \text{ has $p$-torsion} \}.
\end{equation}

We speculate that this set can be shown to be empty in some cases. There is a Hecke action on all the objects considered here and since we are only interested in congruences for $\pi$, we can localize at appropriate maximal ideals. Let $\fr m$ denote the maximal ideal corresponding to $\pi$ and $\tilde {\fr m}$ denote the maximal ideal corresponding to $\tilde \pi$. Since $V_{\fr m}$ pairs nontrivially only with $\tilde V_{\tilde{\fr m}}$, we only need to consider the perfectness of the pairing $L_{\fr m} \otimes \tilde L_{\tilde{\fr m}}$. By the same argument as above, this pairing is perfect whenever 
$$H^b (\partial \SGK, \calm_{\lambda, \calo})_{\fr m} \text{ is torsion-free}.$$
The structure of the boundary cohomology is studied, for example, in \cite[\S V.2]{Scholze} when $F$ is totally real or CM and the sheaf is trivial. If we assume for example that there exists a Galois representation associated to $\pi$ that is residually absolutely irreducible, then this boundary cohomology localized at ${\fr m}$ vanishes.

\medskip

\subsection{The main theorem on congruences}
\label{proof-thmb}

For each prime ideal $\fr l$ of $F$ and $1\le j \le n$, let $T_{\fr l, j}^\circ$ be the modified Hecke operator acting on the integral cohomology groups (see \cite{Ha}).  Let $\calh^\circ$ be the subalgebra of $\End_\calo (L)$ generated by these operators in each of the three cases above. Two automorphic representations $\pi, \pi^\prime$ are said to be congruent modulo $\wp$ if the characters $\chi_1$ and $\chi_2$ of $\calh^\circ$ 
associated to their integral cohomology classes $(\vartheta_{b,\varepsilon}^\circ (\pi)$ and $\vartheta_{b,\eta}^\circ (\pi^\prime))$ are congruent modulo $\wp$. Note that this definition is independent of the choices $\varepsilon$ and $\eta$. The notion of congruence can also be restated in terms of Satake parameters.
Suppose that $\fr l \not \in S_\pi$ and $\wp \not | \, \fr l$, and $q_{\fr l}$ denotes the cardinality of its residue field. Then by \cite[Thm.\,3.21]{Shi}, the local $L$-function of $\pi$ at $\fr l$ is given in terms of a polynomial of the form
$$ 
\sum_{j=0}^n (-1)^j q_{\fr l}^{j(j-1)/2} \chi_1 (T_{\fr l, j}) X^j.
$$
Since the Satake parameters are the inverse roots of this polynomial, we have
$$
\sum_{i_1<i_2 \cdots < i_j} \alpha_{\fr l, i_1} \cdots \alpha_{\fr l, i_j} \ = \ q_{\fr l}^{j(j-1)/2} \chi_1 (T_{\fr l, j}).
$$

Since $T_{\fr l, j}$ and $T^\circ_{\fr l, j}$ differ by a factor supported only at $\fr l$, we can reformulate the congruence condition in terms of Satake parameters of $\pi$ and $\pi^\prime$ as follows:
\begin{equation*}
\sum_{i_1<i_2 \cdots < i_j} \alpha_{\fr l, i_1} \cdots \alpha_{\fr l, i_j}  
\ \ \equiv \ 
\sum_{i_1<i_2 \cdots < i_j} \alpha^\prime_{\fr l, i_1} \cdots \alpha^\prime_{\fr l, i_j} \quad \pmod{\wp} 
\end{equation*}
for all $\fr l \not \in S_\pi \cup S_{\pi^\prime}$ and $\wp \not | \, \fr l$.

\medskip

\begin{theorem}
\label{thm:congruence}
Let $\pi \in {\rm Coh}(G, \lambda, K_f)$ and let $\varepsilon$ be a permissible signature for $\pi.$ 
Let $E,$ $\calo$ and $\wp$ be as before. Suppose that 
$$ 
v_\wp (\Lalg (1, \Ad^0, \pi, \varepsilon)) > 0.
$$

\noindent
{\bf Case I.} If $p \not \in S_1$, then there exists $\pi^\prime$ congruent to $\pi$ mod $\wp$ and 
$\pi^\prime \not \simeq \pi$.

\smallskip

\noindent
{\bf Case II.} Suppose that $\lambda$ is parallel. If $ p \not \in S_2$, then there exists $\pi^\prime$ congruent to $\pi$ mod $\wp$ and $\pi^\prime \not \simeq {^\sigma} \pi$ for any $\sigma \in \Aut(\C)$.

\smallskip

\noindent
{\bf Case III.} Suppose that $\lambda$ is parallel. If $p \not \in S_3$, then there exists 
$\sigma \in \Aut (\C)$ with $\pi' = {^\sigma} \pi$ congruent to $\pi$ mod $\wp$ and $\pi^\prime \not \simeq \pi$.

\smallskip

A priori, we can only say that $\pi^\prime$ contributes to the inner cohomology. If we further assume that $\lambda$ is regular, then $\pi^\prime \in \mathrm{Coh} (G, \lambda, K_f)$.
\end{theorem}

\medskip

\begin{remark}
\label{rem:explanations}
\begin{enumerate}
\item The last assertion concerning a consequence of regularity of $\lambda$ is well-known; 
see Schwermer~\cite{schwermer}. 
\smallskip

\item The obstruction to the converse is that the Hecke congruence module $Q(\calh^\circ; e_1, e_2)$ could be nonzero, while the cohomological congruence module $\calc(L; V_1, V_2)$ is zero, i.e., the converse of 
(\ref{eqn:strong-congruence}) need not hold. See also \cite{ghate-cong}. 

\smallskip

\item If $V=V_1$ in any of the above cases, then $L=L_1$ and by assumption the pairing $L_1 \otimes \tilde L_1$ is perfect. Then the algebraic $L$-value is, {\it a fortiori}, a $\wp$-adic unit and the theorem is vacuously true.
\end{enumerate}
\end{remark}

\medskip
\section{\bf The non-vanishing property}\label{nv-sec}

In this section, we prove Prop.\,\ref{prop:nv-hyp} concerning the non-vanishing of a quantity depending on $\pi_\infty.$ 
Since the non-vanishing is a local condition, we focus on a fixed infinite place. We note that for complex places, this result has already been proved in \cite[\S 5.2]{GHL}. We now consider only the real places.  
Our proof below, which is based on the methods of Sun~\cite{Sun}, can also be modified to give a proof at a complex place which would explicitly give the $K$-types considered in {\em loc.\,cit.}  It is also an interesting problem to consider this issue of nonvanishing from the perspective taken in Harder~\cite{harder-HC} of studying Harish-Chandra modules over $\Z.$ 

Throughout this section we fix once and for all a real place $v \in S_\infty$ and promptly drop it from our notations. 
Let $G=\GL_n(\R)$, $K=\GO(n)$ and $C=\Orth(n)$; and let $K^\circ$ and $C^\circ$ denote the connected component of 
$K$ and $C$ containing the identity. Let $\mu \in \Z^n$ be a dominant integral weight and let $M_\mu$ denote the 
irreducible representation of $G_\C = \GL_n(\C)$ of highest weight $\mu$. We consider irreducible Casselman--Wallach 
representations $\pi$ of $\GL_n(\R)$ such that $\pi$ is unitarizable and tempered and the relative Lie algebra cohomology
$$
H^* (\g, K^\circ; M_\mu \otimes \pi) \not = 0,
$$
where $\g$ is the complexified Lie algebra of $G$. We denote by $\Omega(\mu)$ the set of isomorphism classes of such representations.

Let $\tilde \mu = (-\mu_n, \dots, -\mu_1)$ be the dual weight and we set $ M_\xi = M_{\mu} \otimes M_{\tilde \mu}$. Since we are interested in weights $\mu$ that contribute to cuspidal cohomology, we assume from now on that $\mu$ is pure, i.e., $\tilde \mu = \mu - {\sf w}$ for some integer ${\sf w}$. Fix a nonzero element in the space
$$ \phi_M  \in \Hom_{G_\C} (M_\xi, \id) \not = 0.$$
Let $\pi_\mu \in \Omega(\mu)$ and $\pi_{\tilde \mu} \in \Omega (\tilde \mu)$ and denote by $\pi_\xi = \pi_{\mu} \hat \otimes \pi_{\tilde \mu}$ the completed tensor product of these two representations. We use the integrals $\varTheta $ to construct a nonzero element 
$$
\phi_\pi \in \Hom_G (\pi_\xi, \id),
$$
given by: 
$$
W \otimes \tilde W \ \mapsto \ c^\# (W, \tilde W) \ = \ \frac{\varTheta (W, \tilde W)}{L(1, \pi_\mu \times \pi_{\tilde \mu})}.
$$ 
This map is nonzero since, by \cite[Thm.\,8.5]{Cog}, there exist finitely many $W_i, \tilde W_i$ and $\Phi_i$ such that
$$ 0 \not = L(1, \pi \times \tilde \pi) = \sum_i \Psi_i (1, W_i, \tilde W_i, \Phi_i) = \sum_i \hat \Phi_i (0) \vartheta (W_i, \tilde W_i). $$ 
 In fact, up to multiplication by scalars this is the only nonzero $G$-equivariant map from $\pi_\xi \to \id$; 
see Lemma \ref{1-dim}.
The map $\phi_M \otimes \phi_\pi: M_\xi \otimes \pi_\xi \to \id$ induces the following map at the level of relative Lie algebra cohomology
$$
\Xi: H^{b + t} (\g \times \g, K^\circ \times K^\circ; M_\xi \otimes \pi_\xi ) \longrightarrow H^{b + t} (\g, K^\circ; \id).
$$

\begin{proposition}\label{prop-nv}
Let $\varepsilon$ be a character on $K/K^\circ$ and let $\tilde \varepsilon = \varepsilon_n \varepsilon$ be the dual character, where $\varepsilon_n = (-1)^{n-1}$. The map $\Xi$ restricts to a map on the following eigenspaces, 
$$H^{b + t} (\g \times \g, K^\circ \times K^\circ; M_\xi \otimes \pi_\xi ) (\varepsilon \otimes  \tilde \varepsilon) \longrightarrow H^{b + t} (\g, K^\circ; \id)(\varepsilon_n). $$
The restriction of the above map to the subspace 
$$H^{b} (\g, K^\circ; M_{\mu} \otimes \pi_{\mu} )( \varepsilon) \otimes H^{t} (\g, K^\circ; M_{\tilde \mu} \otimes \pi_{\tilde \mu} )(\tilde \varepsilon) $$
is a nonzero linear map. 
\end{proposition}

The non-vanishing hypothesis is equivalent to the statement of this proposition. Since all the relative Lie algebra complexes involved here are split, the linear map we are interested in is
\begin{equation*}
\begin{split}
\Hom_{K^\circ} &(\wedge^{b} \g/\k, M_{\mu} \otimes \pi_{\mu})(\varepsilon) \otimes \Hom_{K^\circ} (\wedge^{t} \g/\k, M_{\tilde \mu} \otimes \pi_{\tilde \mu})(\tilde \varepsilon) \\
			&\longrightarrow \Hom_{K^\circ \times K^\circ}  (\wedge^{b+t} \g \times \g/\k \times \k, M_\xi \otimes \pi_\xi)(\tilde \varepsilon \otimes \varepsilon)\longrightarrow \Hom_{K^\circ} ( \wedge^{b+t} \g/\k, \id)(\varepsilon_n).
\end{split}
\end{equation*}
Let $\phi_b \in \Hom_{K^\circ} (\wedge^{b} \g/\k, M_{\mu} \otimes \pi_{\mu})(\varepsilon)$ and $\phi_t \in  \Hom_{K^\circ} (\wedge^{t} \g/\k, M_{\tilde \mu} \otimes \pi_{\tilde \mu})(\tilde \varepsilon)$, the image of $\phi_b \otimes \phi_t$ is the composition of all the maps in the first row of the following diagram:
$$\xymatrix@1{
\wedge^{b+t} (\g/\k) \ar[r]^-{\Delta} & \wedge^{b+t} (\g \times \g/\k \times \k) \ar[rr]^-{(\phi_b \otimes \phi_t)^*} & & (M_{\mu} \otimes \pi_{\mu}) \hat \otimes (M_{\tilde \mu} \otimes \pi_{\tilde \mu}) \ar[r]^-{\sim} & M_\xi \otimes \pi_\xi \ar[r]^-{\phi_M \otimes \phi_\pi} & \id \\
& \wedge^b (\g/\k) \otimes \wedge^t (\g/\k) \ar[u] \ar[urr]_-{\phi_b \otimes \phi_t}  & & & & \\
}$$
where $(\phi_b \otimes \phi_t)^*$ extends the map $\phi_b \otimes \phi_t$ by zero.  

Let $\fr t=\fr t_n$ and $\fr b=\fr b_n$ be the Cartan subalgebra and Borel subalgebra for $\g$ considered in \cite{Sun},
 and let $\fr n = \fr n_n$ denote the nilradical of $\fr b$. These subalgebras are chosen in a way that $\fr t^c = \fr t \cap \fr o$ and $\fr b^c = \fr b \cap \fr o$ are the Cartan subalgebra and Borel subalgebra, respectively, for the orthogonal Lie algebra $\fr o$. Let $\fr n^c$ denote the nilradical of $\fr b^c$ and again $\fr n^c = \fr n \cap \fr o$. Let $z_1, z_2 \in \C$, we define a $2\times 2$ complex matrix by the formula
$$
\gamma (z_1, z_2) = \left(\begin{array}{cc} \frac{z_1+z_2}{2} & \frac{z_1 - z_2}{2i} \\ \frac{z_2 - z_1}{2i} & \frac{z_1+z_2}{2} \end{array} \right).
$$ 
Then the Cartan subalgebra $\fr t$ may be identified with $\C^n$, where, if $n=2k$, then the isomorphism is given by 
$$
(y_1, \dots, y_{2k}) \mapsto \text{Diag}(\gamma (y_1,y_2), \dots, \gamma (y_{2k-1}, y_{2k})),
$$
and, when $n=2k+1$, the isomorphism is given by
$$
(y_1, \dots, y_{2k+1}) \mapsto \text{Diag}(y_1, \gamma (y_2,y_3), \dots, \gamma (y_{2k}, y_{2k+1})).
$$
The compact Cartan subalgebra $\fr t^c$ corresponds to the subspace $(-x_1, x_1, \dots, -x_k, x_k)$ when $n$ is even and the subspace $(0, -x_1, x_1, \dots, -x_k, x_k)$ when $n$ is odd. Let $e_i$ denote the character ${\fr t^c}^*$ that sends $(-x_1, x_1, \dots, -x_k, x_k) \mapsto x_i$ when $n$ is even and $(0, -x_1, x_1, \dots, -x_k, x_k) \mapsto x_i$ when $n$ is odd. The $e_i$ form a basis for ${\fr t^c}^*$ and we represent any weight $\lambda \in \tcstar$ by its coefficients $(\lambda_1, \dots, \lambda_k) \in \Z^k$ with respect to this basis. We fix our root system 
$$\Delta_n = \left\{ \begin{array}{cl}
\{\pm e_i \pm e_j \}_{1 \le i < j \le k} & \text{if $n$ is even, and}\\
\{\pm e_i \pm e_j \}_{1 \le i < j \le k} \cup \{ \pm e_i \}_{=1}^{k} & \text{otherwise.}\\
\end{array}\right.
$$
We also fix a set of positive roots
$$\Delta_n^+ = \left\{ \begin{array}{cl}
\{\pm e_i + e_j \}_{1 \le i < j \le k} & \text{if $n$ is even, and}\\
\{\pm e_i + e_j \}_{1 \le i < j \le k} \cup \{e_i \}_{=1}^{k} & \text{otherwise.}\\
\end{array}\right.
$$
Suppose that $\lambda \in \tcstar$ is a dominant weight, we let $\tau_\lambda$ denote the irreducible representation of $\SO(n)$ with highest weight $\lambda$. Fix an element $\delta =\delta_n \in \Orth (n) \setminus \SO(n)$ throughout this section. We take $\delta = \mathrm{Diag}(-1, 1, \dots, 1)$ when $n$ is odd, and 
$$ \delta = \left( \begin{array}{ccccc}
0 & 1 & \\
1 & 0 & \\
& & 1 & \\
& & & \ddots & \\
& & & & 1\\
\end{array} \right)$$
when $n$ is even. This is different from the choice made in \S\ref{compatible-signatures}, but this choice is better for our computations here.
\begin{lemma}\label{lemma-self-dual}
If $n \not \equiv 2 \mod 4$, all the representations $\tau_\lambda$ are self-dual. If $n \equiv 2 \mod 4$, the dual representation $\widetilde \tau_\lambda = \delta \cdot \tau_{\lambda}$. In particular, $\tau_\lambda$ is self-dual if and only if $\lambda_1 =0$.
\end{lemma}
\begin{proof}
We will show that $\widetilde \tau_\lambda \cong \tau_\lambda$ by comparing the highest weights of both representations. The highest weight of $\widetilde \tau_\lambda$ is $-w_{\SO(n)} \lambda$ where $w_{\SO(n)}$ is the longest element in the Weyl group of $\SO(n)$. The lemma follows because $w_{\SO(n)} = -1$ if $n \not \equiv 2 \mod 4$.

If $n \equiv 2 \mod 4$, then $-w_{\SO(n)}$ sends the dominant weight $\lambda=(\lambda_1, \dots, \lambda_k)$ to the dominant weight $ (-\lambda_1, \lambda_2, \dots, \lambda_k)$. One verifies by a simple calculation that the highest weight of $\delta \cdot \tau_\lambda$ is $\delta \cdot \lambda :=  (-\lambda_1, \lambda_2, \dots, \lambda_k)$. This proves the second part of the lemma.
\end{proof}

We consider $\mu$ as a weight for $\fr t$ and let $|\mu|$ be the dominant weight in its Weyl orbit. Let $[|\mu|]$ denote the 
dominant weight corresponding to the 
restriction of $\mu$ to $\fr t^c$. We denote by $\tau_\mu$ the irreducible representation of $\SO(n)$ whose highest weight is $[|\mu|]$. We remark here that since $\mu$ is pure, i.e., $\tilde \mu = \mu - {\sf w}$, we get 
that $[|\tilde \mu|]  = [|\mu|]$. Hence $\tau_{\tilde \mu} = \tau_\mu$.
We also know from \cite{Sun} 
that $\tau_\mu$ is contained in $M_\mu$ with multiplicity one and contains the one dimensional subspace $(M_\mu)^{\fr n}$. Let $\tau_\xi =  \tau_{\mu} \otimes \widetilde \tau_{\mu}$ and we have the following lemma.

\begin{lemma}
The representation $\widetilde \tau_\xi$ occurs with multiplicity one in $M_\xi$. Moreover, every nonzero element of $\Hom_G (M_\xi, \id)$ does not vanish on $\widetilde \tau_\xi$.
\end{lemma}

\begin{proof}
The statement about multiplicity one follows from the corresponding statement for $\mu$ and that $\widetilde M_\mu = M_{\tilde \mu}$ and that $\widetilde \tau_\xi$ equals $\tau_\xi$ or $\delta \cdot \tau_\xi$.  The space $\Hom_G (M_\xi, \id)$ can be canonically identified with $\Hom_G (M_\mu, M_{\mu})$ which is one dimensional by Schur's lemma. Multiplicity one shows that the restriction map is an isomorphism $\Hom_G ( M_\mu, M_{ \mu}) \to \Hom_{\SO(n)} ( \tau_\mu, \tau_{ \mu})$. This proves that any such nonzero linear functional does not vanish on $\tau_\xi$. The second part of the lemma follows from $C$-equivariance. 
\end{proof}

Let $2\rho \in \fr t^*$ denote the sum of all the positive roots for $\fr n$ and let $2\rho^c \in {\fr t^c}^*$ denote the sum of all the positive roots for $\fr n^c$. Denote by $\tau_n$ the irreducible representation of $\SO(n)$ of highest weight $[2\rho]-2\rho^c \in {\fr t^c}^*$. We know that the representation $\tau_n$ occurs with multiplicity one in 
$\wedge^b(\g/\k)$ and by duality $\widetilde \tau_n$ occurs with multiplicity one in $\wedge^t (\g/\k)$. We remark here that when $n=2k$, $\tau_n$ has highest weight $(\pm 2, 4, 6, \dots, 2k)$, where the sign depends on the parity of $k$, and this shows that $\delta \cdot \tau_n \not \cong \tau_n$.

\begin{lemma}
The composition of the maps
$$
\wedge^{b+t} (\g/\k) \ \stackrel{\Delta}{\longrightarrow} \ 
\wedge^{b+t} \left((\g \times \g)/(\k \times \k) \right)  \ \longrightarrow \ \wedge^b (\g/\k) \otimes \wedge^t (\g/\k)
\ \longrightarrow \ \tau_n \otimes \widetilde \tau_n
$$
is nonzero. Here the last two arrows are projection maps onto subspaces. The image of the composition of these three maps is the one dimensional subspace $(\tau_n \otimes \widetilde \tau_n)^{C^\circ}$.
\end{lemma}

\begin{proof}
Let $d=b+t$ and let $X_1, X_2, \dots, X_d$ be a basis for $\g/\k$. The vector $X_1 \wedge X_2 \wedge \cdots \wedge X_d$ is a basis for the one dimensional space $\wedge^{b+t} (\g/\k)$. Its image under the diagonal map is $(X_1,X_1) \wedge (X_2,X_2) \wedge \cdots \wedge (X_d, X_d)$ whose image under the projection map to $\wedge^b(\g/\k) \otimes \wedge^t (\g/\k)$ is 
$$ 
(X_1,X_1) \wedge (X_2,X_2) \wedge \cdots \wedge (X_d, X_d) \mapsto \sum_{\ui =  i_1<i_2 < \cdots < i_b} (X_{i_1} \wedge \cdots \wedge X_{i_b}) \otimes (X_{j_1} \wedge \cdots \wedge X_{j_t}), 
$$
where $\uj = j_1 < j_2 < \cdots <j_t$ is the complement of $\ui$. It is clear that this is a nonzero vector. We now show that this vector projects non-trivially onto the subspace $\tau_n \otimes \widetilde \tau_n$. Let $r=\dim(\tau_n)$ and let 
$k=\dim(\wedge^b (\g/\k))$ and let $v_1, v_2, \dots, v_k $ be a basis for $\wedge^b (\g/\k)$ such that $v_1, \dots, v_r$ is a basis for $\tau_n$. Let $v_1^*, \dots, v_k^*$ be the dual basis for $\wedge^t (\g/\k)$, then $v_1^*,\dots, v_r^*$ is a basis for $\widetilde \tau_n$. We see that up to a nonzero scalar the expression on the right hand side of the above equation is $\sum_{i=1}^k v_i \otimes v_i^*$. Clearly, the projection of this vector onto $\tau_n \otimes \widetilde \tau_n$ is $\sum_{i=1}^r v_i \otimes v_i^*$ which is nonzero. The last statement of the lemma follows from $C^\circ$ equivariance of all the maps.
\end{proof}

Let $\sigma_1$ and $\sigma_2$ be two irreducible representations of $C^\circ \times C^\circ$. Let $\sigma_3 = \sigma_1 \cart \sigma_2$ be the Cartan product of these two representations, i.e., the highest weight of $\sigma_3$ is the sum of the highest weights of $\sigma_1$ and $\sigma_2$. Then $\sigma_2 = \tilde \sigma_1 \prv \sigma_3$, the PRV component of these representations. The PRV product is an irreducible representation whose highest weight is the sum of the highest weight of one representation and the lowest weight of the other. (Here PRV is an acronym for Parthasarathy, Ranga Rao and Varadarajan.)

\begin{proposition}
Every $C^\circ \times C^\circ$-invariant linear functional on $\tilde \sigma_1 \otimes \sigma_3$ is nonzero on the PRV component $\sigma_2$.
\end{proposition}
\begin{proof}
We write each of the representations $\sigma_i = \alpha_i \otimes \beta_i$. Any nonzero functional on $\tilde \sigma_1 \otimes \sigma_3$ is of the form $\phi_1 \otimes \phi_2$ where $\phi_1$ is a linear functional on $\tilde \alpha_1 \otimes \alpha_3$ and $\phi_2$ is a linear functional on $\tilde \beta_1 \otimes \beta_3$. The proof of \cite[Prop.\,2.16]{Sun} shows that $\phi_1$ is nonzero on the PRV component of $\tilde \alpha_1 \otimes \alpha_3$ which is $\alpha_2$. Thus there is a vector $v \in \alpha_2$ such that $\phi_1 (v) \not = 0$. Similarly, there is a vector $w\in \beta_2$ such that $\phi_2 (w) \not = 0$. The map $\phi_1 \otimes \phi_2$ is nonzero on the tensor product $v \otimes w$ and this is an element of $\sigma_2$.
\end{proof}

\begin{definition}
Let $\tau_\mu^+=\tau_\mu \cart \tau_n $ denote the Cartan product of $\tau_\mu$ and $\tau_n$. Then 
$$
\tau_n =  \widetilde \tau_{\mu} \prv \tau_\mu^+.
$$
Let $\tau^+_\xi = \tau^+_{\mu} \otimes \widetilde{\tau^+_{\mu}}$ and we see that $\tau_n \otimes \widetilde \tau_n = \widetilde \tau_\xi \prv \tau_\xi^+$ . 
\end{definition}

We modify the notation from \cite[\S 3]{Sun} and define $W_\mu = \Pi_{S_n} (V_{\tilde \mu})$ and let $W^\infty_\mu$ the Casselman--Wallach smooth globalization of $W_\mu$. We know that $W^\infty_\mu$ is unitarizable and tempered and that the relative Lie algebra cohomology 
$$H^* (\g, K^\circ; M_\mu \otimes W^\infty_\mu) \not = 0.$$
That is, $W^\infty_\mu$ is in $\Omega(\mu)$ and is isomorphic to $\pi_\mu$ up to a twist by the sign character. In any case, we see that $\pi_\xi = W^\infty_\xi := W^\infty_{\mu} \hat \otimes W^\infty_{\tilde \mu}$. The irreducible representation $\tau^+_\mu$ of $\SO(n)$ occurs with multiplicity one in $W_{\tilde \mu}^\infty$ \cite[Lemma 3.3]{Sun}.  The analogous statement for ${\tilde \mu}$ says that $\tau_{\tilde \mu}^+ = \tau_\mu^+$ occurs with multiplicity one in $W_\mu^\infty$. It follows that $\tau^+_\xi$ occurs with multiplicity one in $W^\infty_\xi$.

\begin{lemma}\label{1-dim}
Any nonzero $G$-equivariant linear functional on $W^+_\xi$ does not vanish on $\tau^+_\xi$. For any irreducible Casselman--Wallach representation $\pi$ of $G$, we have
$$\dim\ \Hom_G ( \pi \otimes \tilde \pi , \id) = 1.$$
\end{lemma}

\begin{proof}
A nonzero $G$-equivariant linear functional is non-degenerate. The first part follows since  $\tau^+_{\mu}$ and $\widetilde{\tau^+_{\mu}}$ have multiplicity one.
The second part is just a version of Schur's lemma in this context. 
\end{proof}

\smallskip

\begin{proof}[\bf Proof of Prop.\,\ref{prop-nv}]
Let $\phi_b$ be the composition of the maps
$$
\wedge^b(\g/\k) \ \to \ \tau_n \ \to \  
\widetilde \tau_{ \mu} \otimes \tau_{\mu}^+ \ \to \ M_{ \mu} \otimes W_{ \mu}^\infty,
$$
and similarly we let $\phi_t$ to be the composition of the maps
$$
\wedge^t(\g/\k) \ \to \ \widetilde \tau_n  \ \to \  \tau_{\mu} \otimes \widetilde{\tau_{ \mu}^+} \ \to \ 
M_{\tilde \mu} \otimes W_{\tilde \mu}^\infty.
$$
The maps above are well-defined since $\widetilde \tau_\mu = \tau_\mu$ or $\delta \cdot \tau_\mu$ and both these representations occur with multiplicity one in $M_\mu$. That the image  
of $\phi_b \otimes \phi_t$ in $H^{b + t} (\g, K^\circ; \id)$, denoted $\Xi( \phi_b \otimes \phi_t)$, is nonzero follows from the results in this section and the following commutative diagram:

{\small
$$
\xymatrix@1{
\wedge^{b+t} (\g/\k) \ar[r]^-{\Delta} & \wedge^{b+t} (\g \times \g/\k \times \k) \ar[rr]^-{(\phi_b \otimes \phi_t)^*} &  & (M_{ \mu} \otimes \pi_{ \mu}) \hat \otimes (M_{\tilde \mu} \otimes \pi_{\tilde \mu}) \ar[r]^-{\sim} & M_\xi \otimes \pi_\xi \ar[r]^-{\phi_M \otimes \phi_\pi} & \id \\
& \wedge^b (\g/\k) \otimes \wedge^t (\g/\k) \ar[u] \ar[dd] \ar[urr]_-{\phi_b \otimes \phi_t}  & &  & & \\
& &  &  \widetilde \tau_\xi \otimes \tau_\xi^+ \ar[uur] & & \\
& \tau_n \otimes \widetilde \tau_n &  &  \tau_n \otimes \widetilde \tau_n \ar[u] & & \\
& &  (\tau_n \otimes \widetilde \tau_n)^{C^\circ} \ar[ul] \ar[ur] & & & \\
}$$}

\noindent
When $n$ is odd, the maps $\phi_b$ and $\phi_t$ are generators of the one dimensional vector spaces $H^{b} (\g, K^\circ; M_{\mu} \otimes \pi_{\mu} )( \varepsilon)$ and  $H^{t} (\g, K^\circ; M_{\tilde \mu} \otimes \pi_{\tilde \mu} )(\tilde \varepsilon) $ respectively. The proposition follows.

\smallskip

Now suppose that $n$ is even, one can verify that $\delta \cdot \phi_b$ and $\delta \cdot \phi_t$ are the composition of the natural maps
$$
\wedge^b(\g/\k) \ \to \ \delta \cdot \tau_n \ \to \  \delta \cdot \widetilde \tau_{ \mu} \otimes \delta \cdot \tau_{\mu}^+ \ \to \ 
M_{ \mu} \otimes W_{ \mu}^\infty,$$
and 
$$
\wedge^t(\g/\k) \ \to \ \delta \cdot \widetilde \tau_n  \ \to \ \delta \cdot \tau_{\mu} \otimes \delta \cdot \widetilde{\tau_{ \mu}^+} 
\ \to \ M_{\tilde \mu} \otimes W_{\tilde \mu}^\infty.
$$
Since the $K$-types are not compatible $\Xi (\delta \cdot \phi_b \otimes \phi_t) =0$ and $\Xi (\phi_b \otimes \delta \cdot \phi_t) = 0$. We see that $\phi_{b, \varepsilon} = \phi_b + \varepsilon(\delta) \delta \cdot \phi_b$ and $\phi_{t, -\varepsilon} = \phi_t - \varepsilon(\delta) \delta \cdot \phi_t$ are generators of the one dimensional vector spaces $H^{b} (\g, K^\circ; M_{\mu} \otimes \pi_{\mu} )( \varepsilon)$ and  $H^{t} (\g, K^\circ; M_{\tilde \mu} \otimes \pi_{\tilde \mu} )(-\varepsilon) $ respectively. Since $\Xi (\phi_{b,\varepsilon} \otimes \phi_{t, \varepsilon}) = 0$, we see that $ \Xi (\phi_b \otimes \phi_t) = - \delta \cdot \Xi(\phi_b \otimes \phi_t) $. Finally, we have
\begin{equation*}
\begin{split} 
\Xi (\phi_{b,\varepsilon} \otimes \phi_{t, -\varepsilon}) & \ = \ \Xi (\phi_b \otimes \phi_t) - \delta \cdot \Xi(\phi_b \otimes \phi_t) + \varepsilon(\delta) [\Xi (\delta \cdot \phi_b \otimes \phi_t) - \Xi (\phi_b \otimes \delta \cdot \phi_t)]  \\
										  & \ = \ \Xi (\phi_b \otimes \phi_t) - \delta \cdot \Xi(\phi_b \otimes \phi_t) \\
										  & \ = \ 2\cdot \Xi( \phi_b \otimes \phi_t)
\end{split}
\end{equation*}
which is nonzero. We also remark that this quantity is independent of $\varepsilon$.
\end{proof}

\end{document}